\documentclass[a4paper,12pt]{amsart}
\usepackage[english]{babel}
\usepackage[T1]{fontenc}
\usepackage[latin1]{inputenc}
\usepackage{amssymb}
\usepackage{amsmath}
\usepackage{amsthm}
\usepackage{amscd}
\usepackage{mathrsfs}
\usepackage{bm}
\usepackage{tikz-cd}
\usepackage{pgf,tikz}
\usepackage{mathrsfs}
\usetikzlibrary{arrows}
\usepackage{pgfplots}
\pgfplotsset{compat=1.15}
\usepackage{float}
\usepackage{verbatim}
\usepackage[all]{xy}
\usepackage[cal=boondoxo]{mathalfa}
\usepackage[numbers]{natbib}         % bibliography style
\usepackage[colorlinks]{hyperref}    % better urls in bibliography
\usepackage{color}                   %May be necessary if you want to color links
\usepackage{hyperref}
\usepackage{adjustbox}

\hypersetup{
    colorlinks=true,          %set true if you want colored links
    linktoc=all,              %set to all if you want both sections and subsections linked
    linkcolor=blue,           %choose some color if you want links to stand out
}

\DeclareMathOperator{\Conf}{Conf}

\DeclareMathOperator{\id}{id}

\DeclareMathOperator{\Sq}{Sq}

\DeclareMathOperator{\Fl}{Fl}

\DeclareMathOperator{\FN}{FN}

\DeclareMathOperator{\Orb}{Orb}
\DeclareMathOperator{\Dec}{Dec}
\DeclareMathOperator{\Tree}{Tree}

\newtheorem{theorem}{Theorem}[section]
\newtheorem{proposition}[theorem]{Proposition}
\newtheorem{corollary}[theorem]{Corollary}
\newtheorem{lemma}[theorem]{Lemma}
\newtheorem*{theorem*}{Theorem}

\newtheorem*{Atiyah-weak}{Weak Atiyah--Sutcliffe Conjecture}
\newtheorem*{Atiyah-strong}{Strong Atiyah--Sutcliffe Conjecture}
\newtheorem*{Atiyah-E3}{$ E_3 $ Atiyah--Sutcliffe Conjecture}
\newtheorem*{Atiyah-weak-real}{Weak real Atiyah--Sutcliffe Conjecture}
\newtheorem*{Atiyah-strong-real}{Strong real Atiyah--Sutcliffe Conjecture}
\newtheorem*{Atiyah-E2-real}{$ E_2 $ real Atiyah--Sutcliffe Conjecture}

\theoremstyle{definition}
\newtheorem{definition}[theorem]{Definition}

\theoremstyle{remark}
\newtheorem{remark}[theorem]{Remark}

\newcommand{\dip}[1]{^{ [{#1}] }}

\newcommand{\CP}{\mathbb{CP}}
\newcommand{\RP}{\mathbb{RP}}
\newcommand{\FM}{FM}

\begin{document}

\author[L. Guerra]{Lorenzo Guerra }
\address{Universit\`a degli studi di Milano-Bicocca}
\email{lorenzo.guerra@unimib.it}
\author[P. Salvatore]{Paolo Salvatore}
\address{Universit\`a di Roma Tor Vergata}
\email{salvator@mat.uniroma2.it}

\begin{abstract}
We show that a certain conjecture by Atiyah and Sutcliffe implies the existence of an $E_3 $-algebra (respectively $E_2$-algebra ) structure on the disjoint union of all complex (respectively real) full flag manifolds modulo symmetric groups. Moreover, we show that these structures are liftings of exotic $ E_3 $ (respectively $E_2$) structures on the free $ E_\infty $-algebras on $BU(1)_+$ (respectively $BO(1)_+$), that do not extend to $ E_4 $ (respectively $ E_3 $) structures. 
We also provide some (co)homological calculations supporting the conjecture.
%We also provide some (co)homological calculations showing that those structures are homotopically distinct from the standard ones.
\end{abstract}
\thanks{The authors acknowledge the 
	MUR Excellence Department Project awarded to the Department of Mathematics, University of Rome Tor Vergata, CUP E83C18000100006. 
} 

\title{The Atiyah--Sutcliffe conjecture and $ E_n $ algebras}
\maketitle

\section{Introduction}

In \cite{Atiyah}, motivated by a problem in mathematical physics stated by Berry and Robbins \cite{Berry-Robbins}, Atiyah studied the existence of a continuous map from the configuration 
space $\Conf_n(\mathbb{R}^3)$  of $n$ distinct points in $\mathbb{R}^3$ to the space of $n$-tuples of linearly independent complex lines in $ \mathbb{C}^n $ that is equivariant with respect to the action of the symmetric group $ \Sigma_n $. By orthogonalization we can identify the range up to equivariant homotopy equivalence to the full (or complete) flag manifold $\Fl_n(\mathbb{C})$.

 %$ \Conf_n(\mathbb{R}^3) \to \Fl_n(\mathbb{C}) $.

%As any set of vectors in $ \mathbb{C}^n $ can be symmetrically orthogonalized, defining such a map is equivalent to providing a continuous choice of an $ n $-tuple of linearly independent complex lines in $ \mathbb{C}^n $ from a configuration in $ \mathbb{R}^3 $.
Starting from $ (p_1,\dots,p_n) \in \Conf_n(\mathbb{R}^3) $, Atiyah 
constructs the vectors
\[
x_{i,j} = \frac{p_j - p_i}{||p_j - p_i||} \in S^2 \subseteq \mathbb{R}^3.
\]
By the stereographic projection, $ S^2 $ is homeomorphic to $ \mathbb{CP}^1 $. The element $ \prod_{j\not=i} x_{i,j} $ then belongs to the symmetric product $ SP^{n-1}(\mathbb{CP}^1) $, which is homeomorphic to $ \mathbb{CP}^{n-1} $ via the space of polynomials of degree at most $ n-1 $ and the Fundamental Theorem of Algebra. In particular, it provides a complex line $ \ell_i = \ell_i(p_1,\dots,p_n) $ in $ \mathbb{C}^n $.

Atiyah proposed the candidate map arising from these $ n $ complex lines, but its well-definition depends on the validity of the following conjecture.
\begin{Atiyah-weak}
The lines $ \ell_1(p_1,\dots,p_n), \dots, $  $ \ell_n(p_1,\dots,p_n) $ are projectively independent for all $ (p_1,\dots,p_n) \in \Conf_n(\mathbb{R}^3) $.
\end{Atiyah-weak}

Under this assumption, there is a map $ A_n \colon \Conf_n(\mathbb{R}^3) \to \Fl_n(\mathbb{C}) $. We recall that $ A_n $ is automatically $ SO(3) $-equivariant, in addition to being $ \Sigma_n $-equivariant. Maps with similar properties have been constructed by different techniques \cite{Atiyah,Atiyah-Bielawski}. However, as we will show in this paper, Atiyah's conjectural map is also compatible with the composition of the Fulton-MacPherson operad.

In \cite{Atiyah-Sutcliffe}, Atiyah and Sutcliffe stated a stronger version of this conjecture. From the isomorphism $ S^2 \cong \CP^1 $, they can identify $ x_{i,j}(p_1,\dots,p_n) $ with a complex line $ \ell_{i,j} $ inside $ \mathbb{C}^2 $. They let $ u_{i,j} $ be a representative of this line and let $ v_i = \prod_{j \not= i} u_{i,j} \in SP^{n-1}(\mathbb{C}^2) \cong \mathbb{C}^n $.

Let $ e_1,\dots,e_n $ be the canonical basis vectors of $ \mathbb{C}^n $. Then, in the exterior power $ \Lambda^n(\mathbb{C}) $,
\[
v_1 \wedge \dots \wedge v_n = \varphi e_1 \wedge \dots \wedge e_n
\]
for some $ \varphi \in \mathbb{C} $ depending on the values $ u_{i,j} $. Explicitly, $ \varphi $ is the determinant of the matrix $ [v_1,\dots,v_n] $ whose $ i^{th} $ column is the expansion of $ v_i $ in the canonical basis. If we fix representatives $ u_{i,j} $ once and for all, then $ \varphi = \varphi(p_1,\dots,p_n) $ becomes a function of $ (p_1,\dots,p_n) $. For $ I = \{i_1,\dots,i_r\} \subseteq \{1,\dots,n\}$ a subset they let $ \varphi_I = \varphi(p_{i_1},\dots,p_{i_r}) $. They then define the normalized determinant function
\[
D \colon \Conf_n(\mathbb{R}^3) \to \mathbb{C}, \quad D(p_1,\dots,p_n) = \frac{\varphi(p_1,\dots,p_n)}{\prod_{I\subseteq \{1,\dots,n\}, |I| = 2} \varphi_I(p_1,\dots,p_n)}.
\]
$ D $ is a well-defined continuous function of $ (p_1,\dots,p_n) $ because $ \varphi $ and $ \prod_I \varphi_I $ are homogeneous of the same degree in the representatives $ u_{i,j} $, hence the ratio does not depend on their choice.

\begin{Atiyah-strong}
For all $ (p_1,\dots,p_n) $,
\[
|D(p_1,\dots,p_n)| \geq 1.
\]
\end{Atiyah-strong}

The strong Atiyah-Sutcliffe conjecture clearly implies the weak one, which is equivalent to $ D(p_1,\dots,p_n) \not= 0 $.

There is numerical evidence that this conjecture is true but, until today, it has been proved only for $ n \leq 4 $ and for configurations of some special type. \cite{Dokovic1,Dokovic2,Malkoun}.

One can formulate a ``real''  version of these conjectures. Starting from $ (p_1,\dots, p_n) \in \Conf_n(\mathbb{R}^2) $, we construct vectors $ x_{i,j} \in S^1 \subseteq \mathbb{R}^2 $ by normalization, as above. By the stereographic projection, $ S^1 $ is isomorphic to $ \mathbb{RP}^1 $. The element $ \prod_{j \not= i} x_{i,j} $ belongs to the symmetric product $ SP^{n-1}(\mathbb{RP}^1) $, which embeds in $ \mathbb{RP}^{n-1} $ via the space of polynomials of degree at most $ n-1 $. In particular, it provides lines $ l_i(p_1,\dots, p_n) $ in $ \mathbb{R}^n $. We remark that this embedding is not a homeomorphism (instead $ \mathbb{RP}^{n-1} $ is homeomorphic to a truncated product, see \cite{Kallel-Salvatore}), and its image is identified with the projectivization of the subspace of polynomials with real roots.
	
One can still define the functions $ \varphi $, $ \varphi_I $, and the normalized determinant $ D $ as before, and state the following conjectures:
\begin{Atiyah-weak-real}
The lines $ l_1(p_1,\dots,p_n), \dots, $ $ l_n(p_1,\dots,p_n) $ are projectively independent for all $ (p_1,\dots,p_n) \in \Conf_n(\mathbb{R}^2) $.
\end{Atiyah-weak-real}
\begin{Atiyah-strong-real}
For all $ (p_1,\dots,p_n) \in \Conf_n(\mathbb{R}^2) $, $ |D(p_1,\dots,p_n)| \geq 1 $.
\end{Atiyah-strong-real}
If the weak real Atiyah--Sutcliffe conjecture is true, then there is a map $ A_n \colon \Conf_n(\mathbb{R}^2) \to \Fl_n(\mathbb{R}) $ to the space of complete unordered flags in $ \mathbb{R}^n $ that is $ SO(2) $- and $ \Sigma_n $-equivariant.

The strong real Atiyah--Sutcliffe conjecture implies the weak one. Moreover, the ``complex'' counterpart of each of these conjectures imply the ``real'' version.

The main goal of this paper is to show that, if the strong version of the Atiyah--Sutcliffe conjectures are true, then in the complex case the maps $ A_n \colon \Conf_n(\mathbb{R}^3) \to \Fl_n(\mathbb{C}) $ extend to an $ E_3 $-structure on the disjoint union of complex unordered complete flag manifolds $ \bigsqcup_{n = 0}^{\infty} \frac{\Fl_n(\mathbb{C})}{\Sigma_n} $ (and in the real case $ A_n \colon \Conf_n(\mathbb{R}^2) \to \Fl_n(\mathbb{R}) $ extend to an $ E_2 $-structure on $ \bigsqcup_{n=0}^{\infty} \frac{\Fl_n(\mathbb{R})}{\Sigma_n} $).

As models for the $ E_m $-operad, we use the Fulton--MacPherson operad $ \FM_m = \{ \FM_m(n) \}_n $, because its arity-$ n $ space $ \FM_m(n) $ is the most closely related to the configuration space $ \Conf_n(\mathbb{R}^m) $ (in fact it is a compactification that contains the configuration space modulo dilations and translations as a dense open subset).
As a preliminary step, that does not require assuming the Atiyah--Sutcliffe conjecture, we show that a suitably weighted version of Atiyah's maps $ A_n \colon \Conf_n(\mathbb{R}^3) \to (\mathbb{CP}^{n-1})^n \subseteq (\mathbb{CP}^\infty)^n = BU(1)^n $ extends, necessarily in a unique way, to a map $ \FM_3(n) \to E(\Sigma_n) \times BU(1)^n $, where $ E(\Sigma_n) $ is a total space of the universal bundle for the symmetric group $ \Sigma_n $ and $ BU(1) $ is the classifying space of the circle $ U(1) $. These maps are compatible with the operad composition in $ \FM_3 $ and endow $ \bigsqcup_n E(\Sigma_n) \times_{\Sigma_n} BU(1)^n $ with an ``exotic'' $ E_3 $-action, distinct from the restriction of the classical $ E_\infty $-action.
The spaces $ E(\Sigma_n) \times_{\Sigma_n} X^n $ are also denoted $ D_n(X) $ in the literature and called extended symmetric powers. $ D_n(X) $ can be interpreted as the unordered configuration space of points in $ \mathbb{R}^\infty $ with labels in $ X $. A similar construction can be done in the real case.

Assuming the strong Atiyah--Sutcliffe conjecture, the $ n $ labels in $ BU(1) = \mathbb{CP}^\infty $ are always projectively independent for points of the image of the map $ \FM_3(n) \to E(\Sigma_n) \times BU(1)^n $. This ensures that the $ E_3 $-structure lifts to the disjoint union of the homotopy fibers of the maps $ D_n(BU(1)) \to BU(n) $, which are homotopy equivalent to $ \Fl_n(\mathbb{C}) / \Sigma_n $. Similarly, there is a lifted $ E_2 $-structure on $ \Fl_n(\mathbb{R}) / \Sigma_n $.

If there is an $ E_3 $-structure as above, let us consider the underlying $ E_1 $ and $ E_2 $-structures. The $ E_1 $-structure is determined by the homotopy-associative and homotopy-commutative product $ \frac{\Fl_n(\mathbb{C})}{\Sigma_n} \times \frac{\Fl_m(\mathbb{C})}{\Sigma_m} \to \frac{\Fl_{n+m}(\mathbb{C})}{\Sigma_{n+m}} $ given by direct sum of flags.
In contrast, the restricted $ E_2 $-structure is more elusive: while the underlying $ E_2 $-structure of the exotic $ E_3 $-structure on $ \bigsqcup_n D_n(BU(1)) $ is equivalent to the classical one, 
there are potentially uncountable many distinct $ E_2 $-structures on $ \bigsqcup_n \frac{\Fl_n(\mathbb{C})}{\Sigma_n} $, since there are uncountable many homotopy classes of $ E_2 $ maps from a cofibrant replacement of $ \mathbb{N} $ to $ \bigsqcup_n BU(n) $.
The proof of this fact uses the technical, but potentially interesting on its own, Lemma~\ref{lem:cellular replacement}, stating that $ \bigsqcup SP^n(D^2) $ is a cellular cofibrant replacement of the $ \FM_2 $-algebra $ \mathbb{N} $. This is shown by means of a certain functor from topological operads to spaces. When applied to $ \FM_2 $, it yields the disjoint union of the symmetric products of the disk, and when applied to $ W\FM_2 $, which is isomorphic to $ \FM_2 $, provides a cellular $ \FM_2 $-algebra structure on that space.
We do not know which of the lifted $ E_2 $-structures is given by the action produced from the Atiyah map.

We also perform some calculation of the Dyer--Lashof operations associated with our $ E_3 $- and $ E_2 $-structures. This is possible because Jana and the first author have fully determined the (co)homology of real and complex complete unordered flag manifolds with field coefficients in \cite{Guerra-Santanil}.
Since Atiyah's map is conjectural, in high dimensions its geometry is unclear, and it is very difficult to do calculations by pairing with geometric cohomology generators. Nevertheless, we managed to use a class based on $ 4 $ points to show that the Dyer--Lashof operations associated with the Atiyah $ E_3 $-structure differ from those associated with the classical one even after passing to group completions.
We also combine low-dimensional calculations with formal algebraic machinery to argue that we do not see homological obstructions to the existence of such $ E_3 $-structure with mod $ 2 $ coefficients.

Finally, we tackle the problem of the extendibility of these $ E_k $-structure. We show, using homotopical methods, that the complex $ E_3 $-structure arising from Atiyah's map does not extend to an $ E_4 $-structure, and that the real $ E_2 $-structure does not extend to an $ E_3 $-structure. The extendibility of these structures on the group completions is currently an open problem. In the last section of this paper we show that there are no obstructions to such extensions in mod $ 2 $ homology.

In addition to this introduction, the paper is structured as follows. In Section 2 we construct the conjectural $ E_3 $-structure ($ E_2 $-structure in the real case), state our homotopical versions of the Atiyah--Sutcliffe conjecture and investigate the logical relations between them and the weak/strong versions of the conjecture. Section 3 is devoted to lifting these actions to complete unordered flag manifolds and the analysis of such lifts. Section 4 contains our Dyer--Lashof calculations. In Section 5 we prove our non-extendibility theorem, and in Section 6 we discuss group completions and their homology.

\section*{Acknowledgements}

The authors would like to thank Andrea Bianchi for discussing the topic of this paper and for proposing the core idea of Theorem~\ref{thm:extendibility E2}.

\section{The $ E_3 $- and $ E_2 $-actions}

Let $ D_n(BU(1)) = E(\Sigma_n) \times_{\Sigma_n} BU(1)^n $ be the $ n $-fold extended symmetric power of $ BU(1) $. If we model $ E(\Sigma_n) $ as $ \Conf_n(\mathbb{R}^\infty) $, this becomes the unordered configuration space of $ n $ points in $ \mathbb{R}^\infty $ with labels in $ BU(1) $. In the following we will also use the Fulton--MacPherson compactification $ \FM_\infty(n) $ of $ \Conf_n(\mathbb{R}^\infty) $, introduced in \cite{Fulton-MacPherson}, as a model for $ E(\Sigma_n) $. We use the $ \infty $ index to stress that we are working with configurations in $ \mathbb{R}^\infty $.

We can extend the definition of Atiyah's map to construct an $ E_3 $-action on the space $ D(BU(1)) = \bigsqcup_n D_n(BU(1)) $.
With this goal in mind, we introduce a generalization of Atiyah's maps, that intuitively corresponds to points in a configuration having weights or multiplicities.
\begin{definition} \label{def:weighted Atiyah map}
Let $ w = (w_1,\dots,w_n) \in \mathbb{N}^n $. The weighted Atiyah map is the function
\[
A_n^w \colon \Conf_n(\mathbb{R}^3) \to BU(1)^n, \quad A_n = (\prod_{j\not=1} x_{1,j}^{w_j}, \dots, \prod_{j \not= n} x_{n,j}^{w_j}) \in SP(\CP^1)^n = BU(1)^n,
\]
where products are understood in $ SP(\CP^1) $.
\end{definition}

We first note that the ``direction'' functions $ x_{i,j} \colon \Conf_n(\mathbb{R}^3) \to S^2 $ defined in the previous section extend to the Fulton-Mac Pherson compactification $ \FM_3(n) $ of $ \Conf_n(\mathbb{R}^3) $. This is immediate from Sinha's description of $ \FM_3(n) $ \cite{Sinha:04}.
This remark implies the following statement.
\begin{proposition} \label{prop:extension}
Let $ w \in \mathbb{N}^n $. The weighted Atiyah map extends uniquely to a continuous function $ FM^3_n \to BU(1)^n $
\end{proposition}
With a forgivable abuse of notation, we still denote the extension by $ A_n^w $.

Let $ n,k_1,\dots,k_n \in \mathbb{N} $ and let $ K = \sum_{i=1}^n k_i $. We define a map
\[
\nu_{n:k_1,\dots,k_n} \colon \FM_3(n) \times \prod_{i=1}^n D_{k_i}(BU(1)) \to D_K(BU(1))
\]
as the following composition:
\begin{gather*}
FM_n^3 \times \prod_{i=1}^n \FM_\infty(k_i) \times_{\Sigma_{k_i}} BU(1)^{k_i} \\
\stackrel{A_n^{(k_1,\dots,k_n)} \times \id}{\to} FM_n^3 \times BU(1)^n \times \prod_{i=1}^n \FM_\infty(k_i) \times_{\Sigma_{k_i}} BU(1)^{k_i} \\
\stackrel{\tau}{\to} \left( \FM_3(n) \times \prod_{i=1}^n \FM_\infty(k_i) \right) \times_{\prod_i \Sigma_{k_i}} \prod_{i=1}^n \left( BU(1) \times BU(1)^{k_i} \right) \\
\hookrightarrow \left( \FM_\infty(k_i) \times \prod_{i=1}^n \FM_\infty(k_i) \right) \times_{\prod_i \Sigma_{k_i}} \prod_{i=1}^n \left( BU(1) \times BU(1)^{k_i} \right) \\
\stackrel{\gamma \times \prod_{i=1}^n \mu_{k_i}}{\to} \FM_\infty(K) \times_{\Sigma_K} BU(1)^K,
\end{gather*}
where $ \tau $ permutes the factors in the obvious way, $ \gamma $ is the operadic composition map in $ \FM_{\infty} $ and $ \mu_k \colon BU(1) \times BU(1)^k \to BU(1)^k $ is given by the factor-wise product with the first element ($ BU(1) \cong SP^\infty(\mathbb{CP}^1) $ is a topological monoid).

In an entire similar way, we define maps $ \nu_{n:k_1,\dots,k_n} $ from $ \FM_2(n) \times  \prod_{i=1}^n D_{k_i}(BO(1)) $ to $ D_K(BO(1)) $, where $ K = \sum_{i=1}^n k_i $, via the obvious real version of the weighted Atiyah map $ A_n^{w,\mathbb{R}} \colon \FM_2(n) \to BO(1)^n $.

In order to prove that the maps $ \nu_{n:k_1,\dots,k_n} $ define an algebra structure over the Fulton MacPherson operad, we need to recall how the direction $ x_{i,j} $ behave with respect to operadic composition.
\begin{lemma} \label{lem:directions}
Let $ 1 \leq d \leq \infty $. Let $ \gamma \colon \FM_d(n) \times \prod_{i=1}^n \FM_d(m_i) \to \FM_d(\sum_i m_i) $ be the composition map of the Fulton--MacPherson operad. Then, for all $ p \in \FM_d(n) $, $ q^{(i)} \in \FM_d(m_i) $, and for all $ 1 \leq i,h \leq n $, $ 1 \leq j \leq m_i $, $ 1 \leq k \leq m_h $ with $ (i,j) \not= (h,k) $, the following equality holds:
\[
x_{\sum_{l=1}^{i-1} m_l + j, \sum_{l=1}^{h-1} m_l + k}(\gamma(p;q^{(1)},\dots,q^{(n)})) = \left\{ \begin{array}{ll}
x_{i,h}(p) & \mbox{if } i \not= h \\
x_{j,k}(q^{(i)}) & \mbox{if } i = h
\end{array} \right.
\]
\end{lemma}
The lemma above is known and can be proved by a direct inspection of Sinha's model for $ \FM_d $. See, for instance, \cite[Section 5.2]{Lambrechts-Volic}.

\begin{proposition} \label{prop:operad action DBU1}
The maps $ \nu_{n:k_1,\dots,k_n} $ define an action of the operad $ \FM_3 $ on $ D(BU(1)) $.
Similarly, the maps $ \nu_{n:k_1,\dots,k_n}^{\mathbb{R}} $ define an action of the operad $ \FM_2 $ on $ D(BO(1)) $.
\end{proposition}
\begin{proof}
We focus on the complex case, as the real case is analogous.

Let $ \mu \colon BU(1) \times BU(1) \to BU(1) $ be the product. We need to prove that the following diagram commutes (here we omit the indices in $ \nu $ to avoid adding weight to notation and we let $ M = \sum_{i=1}^n m_i $, $ L_i = \sum_{j = (\sum_{k=1}^{i-1} m_k) + 1}^{\sum_{k=1}^i m_k} l_j $, $ L = \sum_{j=1}^M l_j $):
\begin{center}
\adjustbox{scale=0.7,center}{%
\begin{tikzcd}%[column sep=2.5cm]
\FM_3(n) \times \prod_{i=1}^n \FM_3(m_i) \times \prod_{j=1}^M (\FM_\infty(l_j) \times_{\Sigma_{l_j}} BU(1)^{l_j}) \arrow{r}{\id_{\FM_3(n)} \times \prod_{i=1}^n \nu} \arrow{d}{\gamma \times \prod_{j=1}^M \id_{\FM_\infty(l_j) \times_{\Sigma_{l_j}} BU(1)^{l_j}}} & \FM_3(n) \times \prod_{i=1}^n \left( \FM_\infty(L_i) \times_{\Sigma_{L_i}} BU(1)^{L_i} \right) \arrow{d}{\nu} \\
\FM_3(M) \times \prod_{j=1}^M \left( \FM_\infty(l_j) \times_{\Sigma_{l_j}} BU(1)^{l_j} \right) \arrow{r}{\nu} & \FM_\infty(L) \times_{\Sigma_L} BU(1)^L
\end{tikzcd}
}
\end{center}

Let $ p \in \FM_3(n) $, $ q^{(i)} \in \FM_3(m_i) $, $ r^{(j)} \in \FM_\infty (l_j) $, $ y_j \in BU(1) $.
By unraveling the definitions, we easily see that the top-right composition applied to $ p \times \prod_{i=1}^n q^{(i)} \times \prod_{j=1}^M r^{(j)} \times_{\Sigma_{l_j}} y_j $ is
\begin{gather*}
\Big( \gamma(p;\gamma(q^{(1)};r^{(1)},\dots,r^{(m_1)}),\dots,\gamma(q^{(n)};r^{(\sum_{i=1}^{n-1}m_i)},\dots,r^{(M)})),\\
(\mu(a_1,y_1),\dots,\mu(a_L,y_L))\Big), \\
\mbox{ where } a_{(\sum_{k=1}^{i-1} m_i) + j} = \prod_{h \not= i} x_{i,h}(p)^{L_h} \prod_{k \not= j} x_{j,k}(q^{(i)})^{l_{i,k}} \in BU(1).
\end{gather*}
Similarly, the left-bottom composition yields
\begin{gather*}
\left( \gamma(\gamma(p;q^{(1)},\dots,q^{(n)});r^{(1)},\dots,r^{(M)}),(\mu(b_1,y_1),\dots,\mu(b_L,y_L))\right), \\
\mbox{ where } b_j = \prod_{k \not= j} x_{j,k}(\gamma(p;q^{(1)},\dots,q^{(n)}))^{l_j} \in BU(1).
\end{gather*}
The $ \FM_3(L) $-factors of the two expressions are equal by the associativity of the operad composition of $ \FM_3 $, while $ a_j = b_j $ by Lemma \ref{lem:directions}.
\end{proof}

We state here a new version of the Atiyah--Sutcliffe conjecture, that guarantees that we can lift the $ E_3 $- (respectively $ E_2 $-) structure, at least homotopically, to $ \bigsqcup_n \overline{\Fl}_n(\mathbb{C}) $ (respectively $ \bigsqcup_n \overline{\Fl}_n(\mathbb{R}) $), the disjoint union of the complete unordered flag manifolds $ \overline{\Fl}_n(\mathbb{C}) = \frac{\Fl_n(\mathbb{C})}{\Sigma_n} $ (respectively $ \overline{\Fl}_n(\mathbb{R}) = \frac{\Fl_n(\mathbb{R})}{\Sigma_n} $). $ \overline{\Fl}_n(\mathbb{C}) $ is homotopy equivalent to the homotopy fiber of the classifying map $ D_n(BU(1)) \to BU(n) $ induced by the standard inclusion $ \Sigma_n \wr BU(1) = \pi_1(D_n(BU(1))) \hookrightarrow U(n) $.
This fiber can be identified with the subspace
\begin{gather*}
F_n = \{ (p_1,\dots,p_n) \times_{\Sigma_n} (x_1,\dots,x_n) \in D_n(BU(1)) = E(\Sigma_n) \times_{\Sigma_n} (\CP^{\infty})^n:\\
(x_1,\dots,x_n) \mbox{ are projectively independent in } \CP^{n-1} \subset \CP^{\infty} \}.
\end{gather*}

Therefore, having an induced $ E_3 $-structure on $ \bigsqcup_n \overline{\Fl}_n(\mathbb{C}) $ is equivalent to the following statement.
\begin{Atiyah-E3}
$ \nu_{n:k_1,\dots,k_n}(\FM_3(n) \times \prod_{i=1}^n F_{k_i}) \subseteq F_K $, where $ K = \sum_{i=1}^n k_i $.
Equivalentely, the $ K $ elements of $ BU(1) = \CP^{\infty} $ obtained via the map $ \nu_{n:k_1,\dots,k_n} $ from any point of $ \FM_3(n) $ and points of $ F_{k_i} $ are always projectively independent.
\end{Atiyah-E3}

One can define $ F_n^{\mathbb{R}} \subseteq D_n(BO(1)) $ similarly to $ F_n $, by requiring that the labels are projectively independent in $ \RP^{n-1} $. The following is the real analog of this conjecture.
\begin{Atiyah-E2-real}
$ \nu_{n:k_1,\dots,k_n}^{\mathbb{R}}(\FM_2(n) \times \prod_{i=1}^n F_{k_i}^{\mathbb{R}}) \subseteq F_K^{\mathbb{R}} $, where $ K = \sum_{i=1}^n k_i $.
Equivalentely, the $ K $ elements of $ BO(1) = \RP^{\infty} $ obtained via the map $ \nu_{n:k_1,\dots,k_n}^{\mathbb{R}} $ from any point of $ \FM_2(n) $ and points of $ F_{k_i}^{\mathbb{R}} $ are always projectively independent.
\end{Atiyah-E2-real}

This statement is of intermediate strength with respect to the other versions of the conjecture, as stated below.
\begin{theorem}
\begin{enumerate}
\item if the strong Atiyah--Sutcliffe conjecture is true, then so is the $ E_3 $ one
\item if the $ E_3 $ Atiyah--Sutcliffe conjecture is true, then so is the weak one
\item if the strong real Atiyah--Sutcliffe conjecture is true, then so is the $ E_2 $ one
\item if the $ E_2 $ real Atiyah--Sutcliffe conjecture is true, then so is the weak one
\end{enumerate}
\end{theorem}
\begin{proof}
\begin{enumerate}
\item As we mentioned above, the functions $ x_{i,j} $ extend continuously from the configuration spaces to the Fulton-MacPherson compactifications. As $ \varphi $ and $ \varphi_I $ (for $ I \subseteq \{1,\dots, n\} $) only depend on $ x_{i,j} $ (and a choice of representatives), they can be defined on $ \FM_3(n) $ with the same construction performed on $ \Conf_n(\mathbb{R}^3) $.

One explicitly sees from the construction that, for a suitable choice of representatives, for every $ I $ with $ |I| = 2 $ and for every $ p \in \FN^3_n $, $ \varphi_I(p) $ has the form
\[
	\varphi_I(p) = \det \left[ \begin{array}{cc}
1 & 1 \\
z & -\overline{z}^{-1}
\end{array} \right],
\]
for some $ z \in \mathbb{C} \cup \{ \infty \} $.
In particular, $ \varphi_2(p) \not= 0 $ for all $ p \in \FM_3(n) $.
Consequently, the ratio $ \varphi/\varphi_2 $ is well-defined on all $ \FM_3(n) $, does not depend on the chosen representatives, and provides a continuous extension of the Atiyah--Sutcliffe determinant $ D $ to $ \FM_3(n) $.

If the strong Atiyah--Sutcliffe conjecture is true, then $ |D(p)| \geq 1 $ for all $ p \in \Conf_n(\mathbb{R}^3) $. By continuity and the density of $ \Conf_n(\mathbb{R}^3) $ in $ \FM_3(n) $, $ |D(p)| \geq 1 $ for all $ p \in \FM_3(n) $, and in particular $ \varphi(p) \not= 0 $. It follows that $ A^{(1,\dots,1)}_n(p) $ consists of projectively independent points.

Recall that there is a multiplication-preserving homeomorphism between $ \CP^{\infty} $ and the projectivization of the space of complex polynomials.
For all $ m \in \mathbb{N} $, we consider the distinguished configuration $ s^{(n)} = (s_0,\dots,s_{n-1}) \in \Conf_n(\mathbb{R}^3) $, where $ s_i = (0,0,i) $. We observe that $ A^{(1,\dots,1)}_n(s_n) $ is represented by the standard basis $ (1,t,t^2,\dots,t^{n-1}) $ of the space $ \mathbb{C}[t]_{\leq n-1} $ of polynomials of degree $ \leq n-1 $.
Let $ k_1,\dots,k_n \in \mathbb{N} $, let $ K = \sum_{i=1}^n k_i $ and let $ p \in \FM_3(n) $. Let $ A^{(k_1,\dots,k_n)}_n(p) = ([f_1(t)],\dots,[f_n(t)]) $, for some polynomials $ f_n(t) \in \mathbb{C}[t] $.
By Proposition \ref{prop:operad action DBU1},
\begin{gather*}
A^{(1,\dots,1)}(\gamma(p;s^{(k_1)},\dots,s^{(k_n)})) =\\ ([f_1(t)],[f_1(t)t],\dots,[f_1(t)t^{k_1-1}],[f_2(t)],[f_2(t)t] \dots, [f_n(t)t^{k_n-1}]).
\end{gather*}
Since these polynomials must be linearly independent, they must constitute a basis for $ \mathbb{C}[t]_{K-1} $. If $ r_i = (q^{(i)},(y_{i_1},\dots,y_{i,k_i})) \in F_{k_i} $, then $ y_{i,1},\dots,y_{i,k_i} $ must be represented by a basis of $ \mathbb{C}[t]_{\leq k_i-1} $, so each monomial $ t^j $ (for $ 0 \leq j < k_ i$) must be a linear combination of them. Consequently, each polynomial of the form $ f_i(t)t^j $ (for $ 1 \leq i \leq n $ and $ 0 \leq j < k_i $) must be a linear combination of the labels of $ \nu_{n:k_1,\dots,k_n}(p,r_1,\dots,r_n) $. Thus, these labels generate $ \mathbb{C}[t]_{\leq K-1} $ and thus are projectively independent by dimensional reasons.
\item The weak Atiyah--Sutcliffe conjecture is a particular case of the $ E_3 $ one when $ k_1=\dots k_n = 1 $.
\item It is proved with the same argument used for $ (1) $.
\item It is proved with the same argument used for $ (2) $.
\end{enumerate}
\end{proof}

\section{The underlying $ E_2 $-structure on $ D(BU(1)) $ and its lifts to unordered flag manifolds}

$ D(BU(1)) $ and $ D(BO(1)) $, as shown in the previous section, admit an exotic action of the $ 3 $- and $ 2- $dimensional Fulton-Mac Pherson operad, respectively, induced by Atiyah's construction. There are also ``classical'' $ E_\infty $-structures on $ D(BU(1)) $ and $ D(BO(1)) $ (which are the free non-unital $ E_\infty $-algebras generated by $ BU(1) $ and $ BO(1) $), that restricts to an $ E_3 $-structure on $ D(BU(1)) $ and to an $ E_2 $-structure on $ D(BO(1)) $.
We first compare the two underlying $ E_2 $-structures on $ D(BU(1)) $ and $ E_1 $-structures on $ D(BO(1)) $.

\begin{proposition} \label{prop:E2 same}
	The underlying $ E_2 $-structures of the classical and Atiyah $ E_3 $-structures on $ D(BU(1)) $ are homotopic. Similarly, the underlying $ E_1 $-structures of the classical and Atiyah $ E_2 $-structures on $ D(BO(1)) $ are homotopic.
\end{proposition}
\begin{proof}
	We provide a proof in the complex case, as the real case is entirely similar.
	Consider a function $ f \colon S^1 \to S^2 $. We can restrict weighted Atiyah maps along $ f $ by observing that the vectors $ x_{i,j} $ constructed from a configuration $ (p_1,\dots,p_n) \in \Conf_n(\mathbb{R}^3) $ belong to $ S^1 \subset S^2 $ if $ p_1,\dots,p_n \in \mathbb{R}^2 \subset \mathbb{R}^3 $. Explicitly, for a weight $ w $, we define $ A_n^{w,f} \colon FM_n^2 \to BU(1)^n $ by
	\[
	A_n^{w,f} = (\prod_{j\not=1} f(x_{1,j})^{w_j}, \dots, \prod_{j\not=n} f(x_{n,j})^{w_j}) \in (\mathbb{CP}^{\sum_j w_j})^n \subseteq (\mathbb{CP}^\infty)^n = BU(1)^n.
	\]
	
	We can replace $ A_n^w $ with $ A_n^{w,f} $ in the defintion of $ \nu_{n:k_1,\dots,k_n} $ of Section~3 and obtain modified maps
	\[
	\nu_{n:k_1,\dots,k_n}^{(f)} \colon FM_n^{(2)} \times \prod_{i=1}^n D_{k_i}(BU(1)) \to D_K(BU(1)).
	\]
	
	Note that the structural maps of the $ E_2 $ structure obtained as restriction of the Atiyah $ E_3 $ structure are $ \nu_{n:k_1,\dots,k_n}^{(i)} $, where $ i \colon S^1 \to S^2 $ is the standard embedding, while the restriction of the classical $ E_3 $ structure is given by $ \nu_{n:k_1,\dots,k_n}^{(1_{N})} $, where $ 1_N $ is the constant map equal to the north pole of $ S^2 $.
	
	There is a homotopy $ H \colon [0,1] \times S^1 \to S^2 $ between $ i $ and $ 1_N $ given by
	\[
	H(t,x) = \frac{ti(x) + (1-t)N}{||ti(x)+(1-t)N||}.
	\]
	Let $ h_t(x) = H(t,x) $. The maps $ \nu_{n:k_1,\dots,k_n}^H \colon [0,1] \times FM_n^2 \times \prod_{i=1}^n D_{k_i}(BU(1)) \to D_K(BU(1)) $ defined as
	\[
	\nu_{n:k_1,\dots,k_n}^H(t,x)= \nu_{n:k_1,\dots,k_n}^{h_t}(x)
	\]
	assemble to a homotopy of $ E_2 $-structures between the Atiyah and the classical one.
\end{proof}

A lift of the underlying $ E_2 $-structure to unordered flag manifolds exists without assuming Atiyah's conjecture. As a prelimiary step to prove this fact, we state a technical lemma.

To proof the following lemma, we recall that the subspace of decomposables $ \Dec_{\mathcal{O}}(X) $ of an algebra $ X $ over an operad $ \mathcal{O} $ consists of the elements of $ X $ that arise by an application of operations in $ \mathcal{O} $ of arity at least $ 2 $. An operad is reduced if $ \mathcal{O}(1) $ is a point.

We also recall that the $ W $ construction of a reduced topological operad $ \mathcal{O} $ is another topological operad constructed as follows. We consider the set $ \Tree(n) $ of isomorphism classes of rooted trees with non-leaf and non-root vertices of valence at least $ 3 $, with $ n $ leaves labeled with numbers in $ \{1,\dots,n\} $. We say that a vertex is internal if it neither a leaf nor the root. We say that an edge is internal if it connects two internal vertices. For every such tree $ T $, there is a unique way of orienting its edges such that there exists a unique directed path from every vertex to the root. We will always assume that the edges of $ T $ are oriented this way. We denote with $ r_v $ the number of incoming edges of an internal vertex $ v $. We let $ V(T) $, $ L(T) $ and $ E(T) $ be the sets of internal vertices, leaves and internal edges of $ T $, respectively.
$ W\mathcal{O}(n) $ consists of trees $ T \in \Tree(n) $ with a label in $ \mathcal{O}(r_v) $ attached to each internal vertex $ v \in V(T) $ and a length $ t_e \in [0,1] $ assigned to each internal edge $ e \in E(T) $, with the following identifications: if the length of an edge $ e $ is $ 0 $, this is identified with the tree obtained by collapsing $ e $ and replacing the labels at its vertices with their operadic composition in $ \mathcal{O} $. Operadic composition in $ W\mathcal{O} $ is defined by grafting two trees, keeping vertex labels and edge lengths, and declaring the new edge of the resulting tree to have length $ 1 $. We refer to the classical work of Bordman and Vogt \cite{Boardman-Vogt} for details.

\begin{lemma} \label{lem:cellular replacement}
Let $ \FM_2 $ be the $ 2 $-dimensional Fulton--MacPherson operad. Endow $ \mathbb{N}$ with the $ E_2 $-action induced by its commutative monoid structure. There is a sequence of $ \FM_2 $-algebras $ \{X_n\}_{n \geq 0} $ with $ \FM_2 $-maps $ q_n \colon X_n \to \mathbb{N} $ having the following properties:
\begin{enumerate}
\item $ X_1 = \FM_2(S^0) $
\item for all $ n \geq 1 $, $ X_{n+1} $ is obtained from $ X_n $ by attaching an $ E_2 $-cell of dimension $ 2n $, in the sense of \cite{GKR}
\item the topological pair $ (q_n^{-1}(k),q_n^{-1}(k) \cap \Dec_{\FM_2}(X_n)) $ is homeomorphic to $ (D^{2k-2},S^{2k-3}) $ for all $ k \leq n $
\item $ q_n^{-1}(n+1) $ is homeomorphic to a sphere of dimension $ (2n-1) $
\end{enumerate}
\end{lemma}
\begin{proof}
We construct the sequence recursively. We start by defining
\[
X_1 = \FM_2(S^0) = \bigsqcup_k \frac{\FM_2(k)}{\Sigma_k}
\]
$ 1 $ is obviously satisfied, $ 3 $ holds because both $ \FM_2(0) $ and $ \FM_2(1) $ consist of a single point, and $ 4 $ is satisfied because $ \FM_2(2) $ is homeomorphic to the space of configurations of two antipodal points on $ S^1 $, which is $ \mathbb{RP}^2 \cong S^1 $.

We now assume that we have already constructed a $ \FM_2 $-algebra $ X_n $ satisfying properties $ 3 $, and $ 4 $ above, and we produce $ X_{n+1} $ satisfying properties $ 2 $, $ 3 $ and $ 4 $.
We define $ X_{n+1} $ by attaching a $ 2n $-dimensional disk to $ X_n $ along a homoemorphism between $ q_n^{-1}(n+1) $ and $ \partial(D^{2n}) = S^{2n-1} $. Explicitly, we let $ X_{n+1} $ the following pushout:
\begin{center}
\begin{tikzcd}
\FM_2(D^{2n}) \arrow{r}\drar[phantom, "\urcorner"] & X_{n+1} \\
\FM_2(S^{2n-1}) \arrow[hook]{u} \arrow{r}{\cong} & q_n^{-1}(n+1) \subseteq X_n \hspace{2.2cm} \ar{u}
\end{tikzcd}
\end{center}

We now prove that $ X_{n+1} $ satisfies the desired requirements. $ 2 $ holds by construction. $ 3 $ holds because $ q_{n+1}^{-1}(k) = q_n^{-1}(k) $ if $ k \leq n $ and because $ q_{n+1}^{-1}(n+1) $ is homeomorphic to a cone over $ S^{2n-1} $.

%In order to prove $ 4 $, we let $ G $ be the group of translations and positive dilations in $ \mathbb{R}^2 $ and we let $ \nu \colon \bigsqcup_k \FM^2(k) \times_{\Sigma_k} X_{n+1}^k \to X_{n+1} $ be the operadic action. Recall that the indecomposable $ k $-ary operations of $ \FM^2 $ are those arising from $ F(\mathbb{R}^2,k) $, the unordered configuration space. Thus, since $ \pi_{n+1}^{-1}(1) $ is a point,
%\[
%\pi_{n+1}^{-1}(n+2) = \bigcup_{2 \leq k \leq n+2} \nu \left( \frac{F(\mathbb{R}^2,k)}{G} \times_{\Sigma_n} \bigsqcup_{\sum_{i=1}^k m_i = n+2}X_{n+1}(m_i) \right).
%\]

Before discussing the proof of $ 4 $, we fix some notation.

First, we let $ G $ be the group of translations and positive dilations of $ \mathbb{R}^2 $.%and $ Y_k $ be as in the statement of Lemma \ref{lem:collapsable configurations}. From that lemma, it is enough to prove the following claim.\\
%\textbf{Claim.} There is a homotopy equivalence $ \pi^{-1}_{n+1}(n+2) \sim Y_{n+2} $.

We then define a functor $ P_n $ from the category of topological operads to topological spaces defined by
\[
P_n(\mathcal{O}) = \frac{\mathcal{O}(n)}{\sim}, \mbox{ where } r \circ (s_1,\dots,s_k) \sim r \circ (s_1',\dots,s_k'), \sigma.t \sim t
\]
for all $ r \in \mathcal{O}(k) $, $ s_i,s_i' \in \mathcal{O}(m_i) $, $ t \in \mathcal{O}(n) $ and $ \sigma \in \Sigma_n $.

We also take the topological spaces
\[
{SP'}^n(D^2) = \{[x_1,\dots,x_n] \in SP^n(D^2): \sum_{i=1}^n x_i = 0\}.
\]

Finally, we will use the subcategory $ \mathcal{D}(n) $ of the category of trees defined as follows. Objects are rooted trees with internal vertices of valence at least $ 3 $, at least $ 2 $ leaves and a positive integer weight $ w_{\ell} $ assigned to each leaf $ \ell $, such that $ \sum_{\ell \mbox{ \tiny leaf}} w_{\ell} = n $. We call these ``weighted trees of total weight $ n $''.
Morphisms in $ \mathcal{D}(n) $ are generated under composition by the following two classes of morphisms:
\begin{itemize}
\item We consider tree isomorphisms that preserve the root vertex and the weights (we do not require that they preserve the labels).
\item We consider collapses of internal edges of $ T $ (without changing the leaves weights and labels)
\item For each weighted tree $ T $ and internal vertex $ v \in V(T) $, let $ T_v $ be the subtree spanned by those vertices $ w $ that admit a directed path from $ w $ to $ v $. The leaves of the quotient tree $ T/T_v $ are assigned weights as follows: leaves of $ T $ that do not belong to $ T_v $ survive to $ T/T_v $ and retain their weight; $ T_v $ itself is collapsed into a new leaf with weight equal to the sum of the weights of the leaves of $ T_v $. We consider the quotient maps $ T \to T/T_v $.
\end{itemize}
To clarify this construction, we depict in Figure 1 the isomorphism classes of objects in $ \mathcal{D}(4) $ with the generating morphisms between them. We omit leaf labels and their permutations to keep the figure clean.
\begin{figure}[ht]
\caption{The category $ \mathcal{D}(4) $}
\begin{center}
\begin{tikzpicture}[scale=0.6, line cap=round,line join=round,>=triangle 45,x=1cm,y=1cm]
	\clip(-4.52,-14.8) rectangle (19.72,2.14);
	\draw [line width=1pt] (-1,1)-- (0,0);
	\draw [line width=1pt] (0,0)-- (-0.33,1);
	\draw [line width=1pt] (0.33,1)-- (0,0);
	\draw [line width=1pt] (1,1)-- (0,0);
	\draw [line width=1pt] (0,0)-- (0,-1);
	\draw [line width=1pt] (4,1)-- (5,0);
	\draw [line width=1pt] (6,1)-- (5,0);
	\draw [line width=1pt] (5,0)-- (5,-1);
	\draw [line width=1pt] (9,1)-- (10,0);
	\draw [line width=1pt] (10,0)-- (10,1);
	\draw [line width=1pt] (11,1)-- (10,0);
	\draw [line width=1pt] (10,0)-- (10,-1);
	\draw [line width=1pt] (14,1)-- (15,0);
	\draw [line width=1pt] (15,0)-- (16,1);
	\draw [line width=1pt] (15,0)-- (15,-1);
	\draw [line width=1pt] (-1,-4)-- (-0.33,-5);
	\draw [line width=1pt] (-0.33,-5)-- (0,-6);
	\draw [line width=1pt] (-0.33,-4)-- (-0.33,-5);
	\draw [line width=1pt] (-0.33,-5)-- (0.33,-4);
	\draw [line width=1pt] (1,-4)-- (0,-6);
	\draw [line width=1pt] (0,-6)-- (0,-7);
	\draw [line width=1pt] (4,-4)-- (4.5,-5);
	\draw [line width=1pt] (4.5,-5)-- (5,-4);
	\draw [line width=1pt] (4.5,-5)-- (5,-6);
	\draw [line width=1pt] (5,-6)-- (5,-7);
	\draw [line width=1pt] (5,-6)-- (6,-4);
	\draw [line width=1pt] (9,-4)-- (9.66,-5);
	\draw [line width=1pt] (9.66,-5)-- (9.66,-4);
	\draw [line width=1pt] (10,-6)-- (10.33,-4);
	\draw [line width=1pt] (9.66,-5)-- (10,-6);
	\draw [line width=1pt] (10,-6)-- (11,-4);
	\draw [line width=1pt] (10,-6)-- (10,-7);
	\draw [line width=1pt] (14,-4)-- (14.5,-5);
	\draw [line width=1pt] (14.5,-5)-- (15,-4);
	\draw [line width=1pt] (16,-4)-- (15,-6);
	\draw [line width=1pt] (15,-6)-- (14.5,-5);
	\draw [line width=1pt] (15,-6)-- (15,-7);
	\draw [line width=1pt] (4.5,-14)-- (4.5,-13);
	\draw [line width=1pt] (4.5,-13)-- (5,-10);
	\draw [line width=1pt] (4,-10)-- (3.5,-12);
	\draw [line width=1pt] (3,-10)-- (2.5,-11);
	\draw [line width=1pt] (2.5,-11)-- (2,-10);
	\draw [line width=1pt] (2.5,-11)-- (3.5,-12);
	\draw [line width=1pt] (3.5,-12)-- (4.5,-13);
	\draw [line width=1pt] (12.5,-14)-- (12.5,-13);
	\draw [line width=1pt] (12.5,-13)-- (13.5,-12);
	\draw [line width=1pt] (13.5,-12)-- (13,-11);
	\draw [line width=1pt] (13.5,-12)-- (14,-11);
	\draw [line width=1pt] (12.5,-13)-- (11.5,-12);
	\draw [line width=1pt] (11.5,-12)-- (11,-11);
	\draw [line width=1pt] (11.5,-12)-- (12,-11);
	\draw [-stealth,line width=1pt] (0.36,-3.6) -- (0.38,-1.2);
	\draw [-stealth,line width=1pt] (5.5,-3.6) -- (5.5,-1.2);
	\draw [-stealth,line width=1pt] (10.54,-3.6) -- (10.56,-1.2);
	\draw [-stealth,line width=1pt] (1.32,-3.6) -- (4.14,-1.2);
	\draw [-stealth,line width=1pt] (6.58,-3.6) -- (9.54,-1.2);
	\draw [-stealth,line width=1pt] (9.64,-3.6) -- (1.24,-1.2);
	\draw [-stealth,line width=1pt] (13.98,-3.6) -- (11.26,-1.2);
	\draw [-stealth,line width=1pt] (15.32,-3.6) -- (15.3,-1.2);
	\draw [-stealth,line width=1pt] (2.52,-9.4) -- (0.72,-7);
	\draw [-stealth,line width=1pt] (3.64,-9.4) -- (4,-7);
	\draw [-stealth,line width=1pt] (5.24,-9.4) -- (9,-7);
	\draw [-stealth,line width=1pt] (12.22,-10.3) -- (10.72,-7);
	\draw [-stealth,line width=1pt] (13.2,-10.3) -- (14.4,-7);
	\draw (-1.26,1.52) node[anchor=north west] {$1$};
	\draw (-0.58,1.52) node[anchor=north west] {$1$};
	\draw (0.02,1.52) node[anchor=north west] {$1$};
	\draw (0.7,1.5) node[anchor=north west] {$1$};
	\draw (3.7,1.54) node[anchor=north west] {$3$};
	\draw (5.68,1.54) node[anchor=north west] {$1$};
	\draw (8.64,1.52) node[anchor=north west] {$2$};
	\draw (9.66,1.52) node[anchor=north west] {$1$};
	\draw (10.72,1.52) node[anchor=north west] {$1$};
	\draw (13.68,1.52) node[anchor=north west] {$2$};
	\draw (15.7,1.52) node[anchor=north west] {$2$};
	\draw (-1.28,-3.52) node[anchor=north west] {$1$};
	\draw (-0.62,-3.52) node[anchor=north west] {$1$};
	\draw (0.08,-3.52) node[anchor=north west] {$1$};
	\draw (0.74,-3.52) node[anchor=north west] {$1$};
	\draw (3.74,-3.48) node[anchor=north west] {$2$};
	\draw (4.7,-3.48) node[anchor=north west] {$1$};
	\draw (5.74,-3.48) node[anchor=north west] {$1$};
	\draw (8.7,-3.48) node[anchor=north west] {$1$};
	\draw (9.42,-3.48) node[anchor=north west] {$1$};
	\draw (10.1,-3.48) node[anchor=north west] {$1$};
	\draw (10.76,-3.48) node[anchor=north west] {$1$};
	\draw (13.7,-3.58) node[anchor=north west] {$1$};
	\draw (14.7,-3.58) node[anchor=north west] {$1$};
	\draw (15.72,-3.58) node[anchor=north west] {$2$};
	\draw (1.7,-9.6) node[anchor=north west] {$1$};
	\draw (2.76,-9.6) node[anchor=north west] {$1$};
	\draw (3.78,-9.6) node[anchor=north west] {$1$};
	\draw (4.74,-9.6) node[anchor=north west] {$1$};
	\draw (10.74,-10.54) node[anchor=north west] {$1$};
	\draw (11.7,-10.54) node[anchor=north west] {$1$};
	\draw (12.74,-10.54) node[anchor=north west] {$1$};
	\draw (13.74,-10.54) node[anchor=north west] {$1$};
	\begin{scriptsize}
		\draw [fill=black] (0,0) circle (2.5pt);
		\draw [fill=black] (5,0) circle (2.5pt);
		\draw [fill=black] (10,0) circle (2.5pt);
		\draw [fill=black] (15,0) circle (2.5pt);
		\draw [fill=black] (0,-6) circle (2.5pt);
		\draw [fill=black] (-0.33,-5) circle (2.5pt);
		\draw [fill=black] (5,-6) circle (2.5pt);
		\draw [fill=black] (4.5,-5) circle (2.5pt);
		\draw [fill=black] (10,-6) circle (2.5pt);
		\draw [fill=black] (9.66,-5) circle (2.5pt);
		\draw [fill=black] (15,-6) circle (2.5pt);
		\draw [fill=black] (14.5,-5) circle (2.5pt);
		\draw [fill=black] (2.5,-11) circle (2.5pt);
		\draw [fill=black] (3.5,-12) circle (2.5pt);
		\draw [fill=black] (4.5,-13) circle (2.5pt);
		\draw [fill=black] (12.5,-13) circle (2.5pt);
		\draw [fill=black] (11.5,-12) circle (2.5pt);
		\draw [fill=black] (13.5,-12) circle (2.5pt);
	\end{scriptsize}
\end{tikzpicture}
\end{center}
\end{figure}

%Objects are pairs $ (T,S) $, where $ T \in \Tree(n) $, and $ S $ is a subset of the set of internal vertices of $ T $ covered by a leaf. Automorphisms of $ (T,S) $ in $ \mathcal{D}(n) $ are automorphisms of the tree $ T $ non necessarily preserving the leaf labels, such that 

As the proof is intricate, we split it in separate claims.

\textbf{Claim 1.} $ P_n(\FM_2) $ is homeomorphic to
\[
\partial({SP'}^n(D^2)) = \{ [x_1,\dots,x_n] \in {SP'}^n(D^2) \exists x_i \in \partial(D^2)\}.
\]
To prove this claim, we let $ f_n \colon \FM_2(n) \to \frac{SP^n(\mathbb{R}^2)}{G} $ be the map that merge a cloud of infinitesimally close points into a single point (in other words, $ f_n $ maps an element of $ \FM_2(n) $ to its underlying set of points). The fibers of $ f_n $ are exactly the $ \sim $-equivalent classes and its image is the complement of the thin diagonal $ \Delta $ inside $ \frac{SP^n(\mathbb{R}^2)}{G} $. Hence, it induces a homeomorphism $ \overline{f}_n \colon P_n(\FM_2) \to \frac{SP^n(\mathbb{R}^2) \setminus \Delta}{G} $. By applying a translation in $ G $, an element $ [x_1,\dots,x_n] \in SP^n(\mathbb{R}^2) $ can be transformed into a configuration such that $ \sum_{i=1}^n x_i = 0 $. Moreover, if $ [x_1,\dots,x_n] \notin \Delta $, then at least one point is different from $ 0 \in \mathbb{R}^2 $, thus we can further apply a positive dilation to obtain a configuration in $ \partial({SP'}^n(D^2)) $. This provides the desired homeomorphism.

\textbf{Claim 2.} $ P_n(\FM_2) \cong P_n(W\FM_2) $.\\
This claim follows from the fact, proved by the second author \cite{Salvatore:01}, that there is an isomorphism of topological operads $ \FM_2 \cong W\FM_2 $ and that $ P_n $ is a functor.

\textbf{Claim 3.} Assume that $ X_n $ satisfies properties $ (1), (2), (3) $. Let $ F_{n+1} $ be the functor from $ \mathcal{D}(n+1) $ to topological spaces defined as follows:
\begin{itemize}
\item On an object $ T \in Ob(\mathcal{D}(n+1)) $,
\[
F_{n+1}(T) = \prod_{v \in V(T)} \FM_2(r_v) \times \prod_{l \in L(T)} D^{2(w_\ell-1)}.
\]
In other words, $ F_{n+1}(T) $ is the configuration space of decorations of each internal vertex $ v $ of $ T $ with an element of $ \FM_2(r_v) $, and of each leaf $ \ell $ of $ T $ with a point in $ D^{2(w_\ell -1)} $. To correctly define it, we need a standard order of $ V(T) $. It can be achieved, for instance, by considering for each $ v \in V(T) $ the set of leaves over $ v $, listing the elements of these sets in increasing order, and ordering the resulting lists lexicographically.
\item Let $ \alpha $ be a tree isomorphism. We order the incoming edges of every internal vertex $ v \in V(T) $ as follows: we consider their vertices different from $ v $ and we order them according to the standard order of $ V(T) $. We order the incoming edges of $ \alpha(v) $ similarly. Acting on such edges, $ \alpha $ induces a permutation $ \sigma_{\alpha,v} \in \Sigma_{r_v} $. We let
\[
F_{n+1}(\alpha)(\{x_v\}_{v \in V(T)} \times \{y_\ell\}_{l \in L(T)}) = (\{x_{\alpha^{-1}(v)}.\sigma_{\alpha,v}\}_{v \in V(T')} \times \{y_{\alpha^{-1}(\ell)}\}_{\ell \in L(T')})
\]
\item Let $ e \in E(T) $ be an internal edge. $ F_{n+1} $ transforms the collapse of $ e $ into the morphism defined by applying operadic composition to the factors in $ F_{n+1}(T) $ corresponding to the vertices of $ e $.
\item Let $ v \in V(T) $ be and internal edge. Observe that $ V(T/T_v) = V(T) \setminus V(T_v) $ and that $ L(T/T_v) = L(T) \setminus L(T_v) \cup \{[T_v]\} $.
For $ \{x_v\}_{v \in V(T)} \times \{y_\ell\}_{\ell \in L(T)} \in F_{n+1}(T) $, we define $ y_{[T_v]} $ as follows. First, we apply the tree-wise operadic composition to $ \{x_v\}_{v \in V(T_v)} $ (see, for example, \cite[Section 5.6]{Loday-Vallette} for its definition). This yields an element of $ \FM_2(|L(T_v)|) $. We then apply this operation to $ \{y_{\ell}\}_{\ell \in L(T_v)} $ to obtain an element of $ \Dec_{\FM_2}(X_{\sum_{\ell \in L(T_v)} w_{\ell}}) $. We then identify it with a point in $ \partial(D^{2\sum_{\ell \in L(T_v)} w_{\ell}-2}) $ via the homeomorphisms $ \Dec_{\FM_2}(X_k) \cong S^{2k-3} $.
\end{itemize}
Then $ \varinjlim_{\mathcal{D}(n+1)}(F_{n+1}) \cong q_{n+1}^{-1}(X_n) $.\\
Before proving this claim, we consider the category $ \mathcal{D}'_k(m) $ defined analogously to $ \mathcal{D}(m) $, but we additionally admit trees with a single leaf and we allow only weights $ w_{\ell} \leq k $. Note that $ \mathcal{D}'_n(n+1) = \mathcal{D}(n+1) $.

$ F_m $ extends naturally to $ \mathcal{D}'_n(m) $ for all $ m $ with essentially the same definition (we still denote this functor $ F_m $ with an abuse of notation). We observe that $ \bigsqcup_m \varinjlim_{\mathcal{D}'_n(m)}(F_m) $ has a $ \FM_2 $-algebra structure defined as follows. For all $ x \in \FM_2(k) $, $ y_i \in \varinjlim_{\mathcal{D}'_n(m_i)}(F_{m_i}) $, represent each $ y_i $ by an element $ \tilde{y}_i \in F_{m_i}(T_i) $ for some $ T_i \in \mathcal{D}'_n(m_i) $. Consider the tree $ T $ obtained by grafting $ T_1,\dots, T_k $ onto the leaves of a $ k $-corolla, whose internal vertex is denoted $ \overline{v} $.
Since $ V(T) = \{ \overline{v} \} \sqcup \bigsqcup_{i=1}^k V(T_i) $ and $ L(T) = \bigsqcup_{i=1}^k L(T_i) $, the product $ x \times \prod_{i=1}^k y_i $ defines an element $ z \in F_{\sum_i m_i}(T) $. We define $ x(y_1,\dots,y_k) $ as the image of $ z $ inside $ \varinjlim_{\mathcal{D}'_n(\sum_i m_i)}(F_{\sum_i m_i}) $.
It can be checked directly from the universal property of colimits that this does not depend on the choice of representatives and that this is compatible with composition in $ \FM_2 $.

We now prove that this $ \FM_2 $-algebra is isomorphic to $ X_n $. This is strictly stronger than Claim 3. We proceed by induction on $ n $.
In the base case ($ n = 1 $) all weights are equal to $ 1 $. Disk factors corresponding to weight $ 1 $ consists of a single point. Therefore, for all $ T \in \mathcal{D}'_1(m) $, $ F_m(T) $ consists of the configuration space of decorations of internal vertices of $ T $ with an element of the Fulton--MacPherson operad of the correct arity. $ \varinjlim_{\mathcal{D}'_1(m)}(F_m) $ is thus isomorphic to $ \frac{FM^2(m)}{\Sigma_m} $, with universal maps $ u_T \colon F_m(T) \to \frac{FM^2(m)}{\Sigma_m} $ given by taking tree-wise compositions and passing to the quotient.

Thus we can now assume that $ n > 1 $ and that $ \bigsqcup_m \varinjlim_{\mathcal{D}'_{n-1}(m)}(F_{m}) \cong X_{n-1} $.
$ \mathcal{D}'_{n-1}(m) $ is a subcategory of $ \mathcal{D}'_n(m) $ and $ F_m $ agrees on these categories. Hence, there is a natural morphism
\[
\alpha_n \colon X_{n-1} \cong \bigsqcup_m \varinjlim_{\mathcal{D}'_{n-1}(m)}(F_{m}) \to \bigsqcup_m \varinjlim_{\mathcal{D}'_{n}(m)}(F_{m}).
\]
There is also a natural morphism
\[
\beta_n \colon \FM_2(D^{2n-2}) \to \bigsqcup_m \varinjlim_{\mathcal{D}'_{n}(m)}(F_{m})
\] defined by restricting $ F_m $ to the full subcategory $ \mathcal{D}''_n(m) $ of $ \mathcal{D}'_n(m) $ containing trees with all weight equal to $ n $.
The restrictions of $ \alpha_n $ and $ \beta_n $ to $ \FM_2(S^{2n-3}) $ are equal by construction.

To prove that $ \bigsqcup_m \varinjlim_{\mathcal{D}'_{n}(m)}(F_{m}) $ is the pushout of $ X_{n-1} $ and $ \FM_2(D^{2n-2}) $ along $ \FM_2(S^{2n-3}) $ (which is $ X_n $ by definition), we check that it satisfies the corresponding universal property.

Let $ Y $ be a $ \FM_2 $-algebra with maps $ \varphi \colon X_{n-1} \to Y $ and $ \psi \colon \FM_2(D^{2n-2}) \to Y $ agreeing on $ \FM_2(S^{2n-3}) $.
Then we can map $ D^{2k-2} $ with $ k < n $ to $ Y $ by first embedding $ D^{2k-2} $ into $ X_{n-1} = \bigsqcup_m \varinjlim_{\mathcal{D}'_{n-1}(m)} ( F_m ) $ via the tree with a single leaf of weight $ k $, and then mapping $ X_{n-1} $ to $ Y $ via $ \varphi $. We map $ D^{2n-2} $ to $ Y $ via $ \psi $. These functions map the decoration of leaves of trees in $ \mathcal{D}'_n(m) $ into $ Y $.

For all trees $ T \in \mathcal{D}_n(m) $, we define $ \chi_T \colon F_m(T) \to Y $ by first taking the tree-wise composition of the decorations of the internal vertices of $ T $, and then apply the resulting operation to the images of the decorations of its leaves in $ Y $.
It is straightforward to check that the maps $ \chi_T $ is a cocone over $ F $ and, consequently, induce a morphism $ \varinjlim{\chi} \colon \bigsqcup_m \varinjlim_{\mathcal{D}'_{n}(m)}(F_{m}) \to Y $. $ \varinjlim{\chi} \circ \alpha_n = \varphi $ and $ \varinjlim{\chi} \circ \beta_n = \psi $ by construction.
Let $ \chi' \colon \bigsqcup_m \varinjlim_{\mathcal{D}'_{n}(m)}(F_{m}) \to Y $ be another  $ \FM_2 $-algebra map such that $ \chi' \circ \alpha_n = \varphi $ and $ \chi' \circ \beta_n = \psi $. The last two identities imply that $ \chi $ and $ \chi' $ have the same restriction on objects $ T \in \bigsqcup_m \left(\mathcal{D}'_{n-1}(m) \cup \mathcal{D}''_n(m) \right) $. Compatibility with the action of $ \FM_2 $ implies that the restriction of $ \chi' $ to objects of $ \bigsqcup_m \mathcal{D}'_n(m) $ preserves grafting of weighted trees on corollas. Since every weighted tree in $ \bigsqcup_m \left( \bigsqcup_m \mathcal{D}'_n(m) \right) $ can be obtained from objects of $ 
\mathcal{D}'_{n-1}(m) \cup \mathcal{D}''_n(m) $ via such grafting operations, $ \chi' = \varinjlim \chi_T  = \chi $.

\textbf{Claim 4.} If $ X_n $ satisfies conditions $ (1) $, $(2) $ and $(3) $, then $ P_{n+1}(W\FM_2) \cong \varinjlim_{\mathcal{D}(n+1)} F_{n+1} $.\\
We prove this claim by induction on $ n $, the base case ($ n = 1 $) being trivial.

We assume that we have already constructed homeomorphisms $ \varphi_k \colon P_{k}(W\FM_2) \to \varinjlim_{\mathcal{D}(k)} F_{k} $ for all $ k \leq n $ and we define $ \varphi_{n+1} \colon P_{n+1}(W\FM_2) \to \varinjlim_{\mathcal{D}(n+1)} F_{n+1} $ as follows. Let $ x \in W\FM_2(n+1) $ be the element corresponding to a rooted tree $ T \in \Tree(n+1) $, with internal vertices $ v $ decorated with $ x_v \in \FM_2(r_v) $ and internal edges $ e $ of length $ t_e $. We let $ \overline{v} \in V(T) $ be the internal vertex closest to the root and  $ e_1,\dots,e_r $ be the ingoing edges of $ \overline{v} $. We let $ T_i $ be the sub-tree of $ T $ over $ e_i $ (with the same leaf weights and vertex decorations as $ T $). Then $ T $ is obtained by grafting $ T_1,\dots, T_r $ onto an $ r $-corolla (with internal vertex $ \overline{v}$), up to a permutation $ \sigma $ of the leaf labels. Let $ x|_{T_i} $ be the operation of $ W\FM_2 $ supported on $ T_i $, with vertices decoration and edge length borrowed from $ x $. Then $ \varphi_{|L(T_i)|}(x_i) \in \varinjlim_{\mathcal{D}(|L(T_i)|)} F_{|L(T_i)|} \cong q_{|L(t_i)|}^{-1}(X_{|L(T_i)|-1}) \cong S^{2|L(T_i)|-3} $. For all $ 0 \leq \lambda \leq 1 $, we can rescale it to a factor $ \lambda $ to obtain a point of the disk $ \lambda \varphi_{|L(T_i)|}(x|_{T_i}) \in D^{2|L(T_i)|-2} $. Let $ C_T \in \mathcal{D}_{n+1}(n) $ be the $ r $-corolla with weights equal to $ |L(T_1)|, \dots, |L(T_r)| $. We define $ \tilde{\varphi}_{n+1}(x) $ as the element of $ F_{n+1}(C_T) $ given by
\[
 x_{\overline{v}} \times ((1-t_{e_1})\varphi_{|L(T_1)|}(x|_{T_1}),\dots,(1-t_{e_r})\varphi_{|L(T_r)|}(x|_{T_r})).
\]
Let $ \overline{x} $ be the image of $ x $ in the quotient space $ P_{n+1}(WFM^2) = \frac{WFM^2(n+1)}{\sim} $. We then define $ \varphi_{n+1}(\overline{x}) $ as the image of $ \tilde{\varphi}_{n+1}(x) $ inside $ \varinjlim_{\mathcal{D}(n+1)} F_{n+1} $.

In order to show that $ \varphi_{n+1} $ is well-defined, we need to check some properties:
\begin{itemize}
\item $ \Sigma_{n+1} $ acts on $ W\FM_2 $ by permuting leaf labels. As the construction of $ \varphi_{n+1} $ does not depend on them, this map must be $ \Sigma_{n+1} $-invariant.
\item We need to show that, if $ t_{e_i} = 0 $ for some $ i $ and $ y $ is obtained from $ x $ by collapsing the edge $ e_i $ of $ T $ and replacing the decorations of $ x $ at the vertices of $ e_i $ with their operadic composition, then $ \varphi_{n+1}(x) = \varphi_{n+1}(y) $. Under that assumption, $ (1-t_{e_i})\varphi_{|L(T_i)|}(x_i) = \varphi_{|L(T_i)|}(x_i) \in S^{2|L(T_i)|-3} $. Let $ u_i $ be the vertex of $ e_i $ different from $ \overline{v} $. Let $ T' $ be the weighted tree (with the two internal vertices $ u_i $ and $ \overline{v} $) obtained by collapsing the $ T_j $s with $ j \not= i $ and the subtrees $ S_1, \dots, S_s $ above $ u_i $.
By the recursive definition of $ \varphi_k $, we see that $ \tilde{\varphi}_{n+1}(x) $ is the image in $ F_{n+1}(C_T) $ of
\[
(x_{u_i},x_{\overline{v}}) \times (\varphi_{|L(S_1)|}(x|_{S_1}),\dots,\varphi_{|L(S_s)|}(x|_{S_s}),\dots, \varphi_{|L(T_r)|}(x|_{T_r})) \in F_{n+1}(T')
\]
by the map induced by collapsing $ e_i $. As this map is defined from the operadic composition, this is exactly $ \varphi_{n+1}(y) $.
\item If $ t_e = 0 $ for some $ e \in \bigcup_i E(T_i) $, then $ \varphi_{n+1}(x) = \varphi_{n+1}(y) $ by our inductive assumption.
\item If $ t_{e_i} = 1 $ for some $ i $, then $ (1-t_{e_i}) \varphi_{|L(T_i)|}(x|_{T_i}) $ is the center of $ D^{2|L(T_i)|-2} $ and does not depend on the part of $ x $ above $ e_i $.
\item If $ t_e = 1 $ for some $ e \in \bigcup_i E(T_i) $, then $ \varphi_{n+1}(x) $ does not depend on the part of $ x $ above $ e $ by our inductive hypothesis.
\end{itemize}
Thus $ \varphi_{n+1} $ is well-defined.

We now explicitly construct the inverse of $ \varphi_{n+1} $. For all $ T \in \mathcal{D}_n(n+1) $ and $ x = (\{x_v\}_v \times \{y_\ell\}_\ell) \in F_{n+1}(T) $, we take the tree-wise composition of all the decorations of the internal vertices of $ T $ to obtain an operation of $ z \in \FM_2(|L(T)|) $. We then construct $ \tilde{\psi}_T \colon F_{n+1}(T) \to P_n(W\FM_2) $ by starting with a corolla with an internal vertex decorated with $ z $, onto which we graft the decorated rooted trees $ \varphi^{-1}_{|w_{\ell}|}(y_{\ell}) $ indexed by $ L(T) $. This provides an element of $ P_{n+1}(W\FM_2) $. It is easily checked that $ \tilde{\psi} \colon F_{n+1} \to P_{n+1}(W\FM_2) $ is a cocone. Therefore, it induces a morphism $ \psi \colon \varinjlim_{\mathcal{D}(n+1)} F_{n+1} \to P_{n+1}(W\FM_2) $.
$ \psi \circ \varphi_{n+1} = \id_{P_{n+1}(W\FM_2)} $ by construction.
We observe that the weighted trees with only one internal vertex span a cofinal subcategory of $ \mathcal{D}_n(n+1) $. Therefore, to prove that $ \varphi_{n+1} \circ \psi = \id_{\varinjlim_{\mathcal{D}(n+1)} F_{n+1}} $, it is enough to check that $ \tilde{\varphi}_{n+1} \circ \psi_T = \id_{F_{n+1}(T)} $ for these weighted trees. This is true by definition.

\textbf{Claim 5.} If $ X_{n+1} $ satisfies conditions $ (1) $, $(2) $ and $ (3) $, then it also satisfies condition $ (4) $.\\
To prove this statement, we combine all the previous claims to obtain a chain of homeomorphisms
\[
\partial({SP'}^{n+1}(D^2)) \cong P_{n+2}(\FM_2) \cong P_{n+2}(W\FM_2) \cong \varinjlim_{\mathcal{D}_n(n+2)}(F_{n+2}) \cong \pi_{n+2}^{-1}(X_{n+1}).
\]
The spaces $ \partial({SP'}^{n+1}(D^2)) $ and, consequently, $ \pi_{n+2}^{-1}(X_{n+1}) $ are homeomorphic to a sphere of the correct dimension by a straightforward modification of \cite[Lemma 5]{Kallel-Salvatore}.
\end{proof}

A consequence of the proof of this technical lemma is the following:
\begin{corollary}
$ \bigsqcup_n SP^n(D^2) $ has the structure of a cofibrant $ \FM_2 $-algebra.
\end{corollary}

\begin{theorem} \label{thm:extendibility E2}
Let $ X_{\infty} = \varinjlim_n X_n $. There is an uncountable set of non-homotopic $ E_2 $-maps $ X_{\infty} \to \bigsqcup_n BU(n) $.
\end{theorem}
\begin{proof}
Let $ q_{\infty} = \varinjlim_n q_n \colon X_{\infty} \to \mathbb{N} $ and $ X_{\infty,n} = q_{\infty}^{-1}(n) $. Let $ \alpha \colon X_{\infty} \to \bigsqcup_n BU(n) $ be an $ E_2 $-map.	
		
We let $ \alpha_n $ be the restriction of $ \alpha $ to $ X_n $. There is only one possible homotopy class for $ \alpha_1 $ because $ BU(1) $ is connected.

Assume that $ \alpha_n $ is already defined. Since $ X_{n+1} $ is obtained by attaching an $ E_2 $-cell of dimension $ 2n-2 $ to $ X_n $, we can extend $ \alpha_n $ to a map $ \alpha_{n+1} \colon X_{n+1} \to \bigsqcup_m BU(m) $ preserving components if and only if $ \pi_{2n-3}(BU(n)) $ is trivial. $ \pi_{2n-3}(BU(n)) \cong \pi_{2n-4}(U(n)) \cong \pi_{2n-4}(U(\infty)) = 0 $ by Bott periodicity, as we are in the stable range. Therefore, an extension $ \alpha_{n+1}$ always exists.

The possible homotopy classes of such extensions are classified by $ \pi_{2n}(BU(n)) $ which, by the same argument, is isomorphic to $ \pi_{2n-3}(BU(\infty)) \cong \mathbb{Z} $.

Passing to the limit, we see that the set of homotopy classes of maps $ \alpha $ is in bijection with $ \mathbb{Z}^{\mathbb{N}} $, which is uncountable.
\end{proof}

Let $ q_{\infty} = \varinjlim_n q_n \colon X_{\infty} \to \mathbb{N} $ and $ X_{\infty,n} = q_{\infty}^{-1}(n) $.
Then $ X_{\infty} = \bigsqcup_{n \in \mathbb{N}} X_{\infty,n} $ and each $ X_{\infty,n} $ is contractible. Assume that a $ E_2 $-map $ \alpha = \bigsqcup_n \alpha_n \colon \bigsqcup_n X_{\infty,n} \to \bigsqcup_n BU(n) $ is given and construct the pullback $ Y $ with the standard map $ DBU(1) \to \bigsqcup_n BU(n) $:
\begin{center}
	\begin{tikzcd}
		Y_{\alpha} \arrow{r} \arrow{d} & D(BU(1)) \arrow{d} \\
		X_{\infty} \arrow{r}{\alpha} & \bigsqcup_n BU(n)
	\end{tikzcd}
\end{center}

As both $ \alpha $ and the right vertical map are $ E_2 $-maps, there is an induced $ E_2 $-action on $ Y_{\alpha} $. Moreover, since $ X_{\infty} $ is cofibrant with contractible components, $ Y_{\alpha} $ is homotopy equivalent to $ \bigsqcup_n F_n $. We conjecture that, if two such maps $ \alpha $ and $ \alpha' $ are non-homotopic as $ E_2 $-morphisms, then the two induced $ E_2 $-actions on $ Y_{\alpha} $ and $ Y_{\alpha'} $ must be non-homotopic as $ E_2 $-maps over $ D(BU(1)) $. If this is true, then there is an inclusion from the set of homotopy classes of $ E_2 $-maps $ X_{\infty} \to \bigsqcup_n BU(n) $ preserving components into the set of equivalence classes of $ E_2 $-structures on $ \bigsqcup_n F_n $ compatible with the $ E_2 $-structure on $ D(BU(1)) $.

\textbf{Conjecture.} There are uncountably many $ E_2 $-structures on $ \bigsqcup_n F_n $ such that the inclusion $ \bigsqcup_n F_n \to D(BU(1)) $ is a homotopy morphism between $ E_2 $-algebras, up to homotopies of $ E_2 $-maps commuting with the morphisms $ \bigsqcup_n F_n \to D(BU(1)) $.

\section{Calculations for the Atiyah $ E_3 $-structure on $ D(BU(1) $}

On the mod $ 2 $ homology of $ E_3 $-algebras there is an action of Kudo--Araki--Dyer--Lashof operations up to order $ 2 $. We perform some partial calculations of this action to show that, even at the level of homology, the two considered $ E_3 $-actions are different.

To avoid ambiguity, in $ H_*(D(BU(1)); \mathbb{F}_2) $, we let $ Q_i $ and $ \tilde{Q}_i $ be the Dyer-Lashof operations (with lower indices) for the classical and the Atiyah $ E_3 $-structure, respectively. Moreover, we let $ * $ be the Pontrjagin product in homology.

We start with a relatively simple calculation that is enough to show that unstably the two structures are different.
\begin{lemma} \label{lem:identity unstable}
In $ H_*(D(BU(1)); \mathbb{F}_2) $, $ \tilde{Q}_2([1]) = Q_2([1]) + (c_1 \odot 1_1)^\vee $.
\end{lemma}
\begin{proof}
We consider the basis of decorated Hopf monomials $ \mathcal{B} $ introduced by Sinha and the authors \cite{GSS} and we let $ \mathcal{B}^\vee = \{ x^\vee: x \in \mathcal{B} \} $ be the dual basis. We also let $ c \in H^2(D_1(BU(1)); \mathbb{F}_2) $ be the (mod $ 2 $ restriction of) the first universal Chern class.

We will also adopt the geometric (co)homology framework developed by Friedman--Medina-Mardones--Sinha \cite{FMS}.

To improve clarity, we split the proof in two distinct claims.

\textbf{Claim 1.} $ \tilde{Q}_2([1]) = Q_2([1]) + \lambda (c_1 \odot 1_1)^\vee $ for some $ \lambda \in \mathbb{F}_2 $.
To prove this, we consider the $ E_\infty $-morphisms $ \iota \colon D(\{*\}) \to D(BU(1)) $ and $ p \colon D(BU(1)) \to D(S^0) $ induced by a map that sends $ * $ to any point of $ BU(1) $ and by the projection $ BU(1) \to \{*\} $, respectively.

If we endow $ D(BU(1)) $ with the Atiyah $ E_3 $-structure and $ D(\{*\}) $ with the restriction of the classical $ E_\infty $-structure, $ p $ is an $ E_3 $-morphism and the composition $ p \circ \iota $ is the identity. Hence that $ \ker(p_*) = (\ker(\iota^*))^\vee $ is the subspace of $ H_*(D(BU(1)); \mathbb{F}_2) $ spanned by linear duals of decorated Hopf monomials where $ c $ appears. In dimension $ 2 $ and component $ 2 $, this is generated by $ (c \odot 1_1)^\vee $.
We deduce that $ \tilde{Q}_2([1]) = Q_2([1]) + \lambda (c \odot 1_1)^\vee $ for some $ \lambda \in \mathbb{F}_2 $.

We remark that Claim 1 is enough to prove that the stabilizations of the two $ E_3 $-structures on $ D_\infty(BU(1)) \times \mathbb{Z} $ are homotopically different, but not to prove the same for that their $ 0 $-component.

\textbf{Claim 2.} $ \lambda = 1 $.
To determine $ \lambda $, we argue geometrically. By choosing $ \mathbb{CP}^\infty $ as a model for  $ BU(1) $, we can identify $ D_2(BU(1)) $ with the unordered labeled configuration space
\[
M = \overline{\Conf}_2(\mathbb{R};\mathbb{CP}^\infty) = \frac{\Conf_2(\mathbb{R}^\infty; \mathbb{CP}^\infty)}{\Sigma_2}.
\]
$ M $ is the direct limit of its finite-dimensional submanifolds
\[
M_n = \overline{\Conf}_2(\mathbb{R}^n; \mathbb{CP}^n).
\]
Assume that $ r_W \colon W \to M $ is a proper finite-codimensional immersion into $ M $ such that the proper manifold over $ M_n $ $ W_n = r_W^{-1}(M_n) $ has trivial mod $ 2 $ boundary in the geometric cochain complex of $ M_n $. Then, as $ H^*(M_n; \mathbb{F}_2) = \varprojlim_{n} H^*(M_n; \mathbb{F}_2) $ and $ [W_n]|_{H^*(M_{n-1}; \mathbb{F}_2)} = [W_{n-1}] $, the limit $ \varprojlim_n [W_n] $ define a mod $ 2 $ cohomology class $ [W] \in H^*(M; \mathbb{F}_2) $. We call $ [W] $ the cohomology class geometrically represented by $ W $.

We let $ W $ be the sub-manifold of $ M $ given by configurations of two points $ p_1 $ and $ p_2 $ with labels $ x_1 $ and $ x_2 $ such that the first projective coordinate of $ x_1, x_2 $ and $ x_3 $ are all $ 0 $.
We let $ r_W \colon W \to M $ be the embedding.
We observe that $ r_W $ is proper and that $ c \odot 1_1 $ is represented by $ W $.

We also observe that $ \tilde{Q}_2([1]) $ is a geometric chain represented by the submanifold $ \widetilde{\Orb}_{(2)} $, diffeomorphic to $ \mathbb{P}^2(\mathbb{R}) $, of configuration of points of the form $ \nu_{2,1,1}(S \times_{\Sigma_2} \{*\}^2) $, where $ S $ is the subspace of configurations of two antipodal points on $ S^2 $ and $ * \in BU(1) $ is any chosen basepoint (compare with \cite[Proposition 4.5]{Sinha:12}).

One can compute $ \lambda =  \langle \tilde{Q}_2[1], c \odot 1_1 \rangle $ by computing the fibered product of $ r_W $ and $ \widetilde{\Orb}_{(2)} $ and counting its cardinality mod $ 2 $. A direct calculation shows that, if we choose $ * = [1:0:0:\dots] $,
\[
\widetilde{\Orb}_{(2)} = \{((p,-p),(j(p),j(\overline{p}^{-1}))): p \in S^2 \},
\]
where $ j \colon S^2 \cong \mathbb{CP}^1 \to \mathbb{CP}^\infty $ is the inclusion.
It intersects the image of $ r_W $ in a single point. Therefore $ \lambda = 1 $.
\end{proof}

\begin{proposition} \label{prop:different}
In the mod $ 2 $ homology of the group completion $ D_\infty(BU(1)) \times \mathbb{Z} $,
\[
\tilde{Q}_2(\tilde{Q}_2([1])*[-2]) \not= Q_2(\tilde{Q}_2([1])*[-2]).
\]
In particular, the classical and Atiyah $ E_3 $-structures on the $ 0 $-component of the stabilization $ D_\infty(BU(1)) $ are different.
\end{proposition}
\begin{proof}
We again split the proof in two claims. We also use the identity granted by the lemma above.

\textbf{Claim 1.} In terms of linear duals with respect to the basis of stabilized decorated Hopf monomials in cohomology,
\begin{align*}
Q_2(\tilde{Q}_2[1] * [-2]) &= (\gamma_2^2 \odot 1\dip{*})^\vee + ((\gamma_1\dip{2})^3 \odot 1\dip{*})^\vee + ((\gamma_1\dip{2})^2 \odot \gamma_1 \odot 1\dip{*})^\vee \\
&+ (\gamma_1^2 c\dip{2} \odot 1\dip{*})^\vee + (\gamma_1 c\dip{2} \odot \gamma_1 \odot 1\dip{*})^\vee + (c\dip{2} \odot \gamma_1\dip{2})^\vee.
\end{align*}

From the Adem relations we deduce that $ Q_2 Q_1 = Q_1 Q_0 = 0 $ and that $ Q_2 Q_0 = Q_0 Q_1 $.
From the compatibility between the Kudo--Araki operations and the Hopf algebra structure of $ H_*(D_\infty(BU(1)); \mathbb{F}_2) $ and the pairing between the Hopf monomial basis of $ H^*(D(BU(1)); \mathbb{F}_2) $ and the Nakaoka basis in $ H_*(D(BU(1)); \mathbb{F}_2) $ we deduce that:
\begin{gather*}
Q_0[-2] = [-4]\\
Q_1[-2] = Q_1Q_0[-1] = 0 \\
0 = Q_1([1]*[-1]) = Q_1[1] * [-2] + [2] * Q_1[-1] \Rightarrow Q_1[-1] = Q_1[1] * [-4] \\
0 = Q_2([1]*[-1]) = Q_2[1] * [-2] + Q_1[1] * Q_1[-1] + [2] * Q_2[-1] \\
\Rightarrow Q_2[-1] = Q_2([1]) * [-4] + Q_0Q_1([1]) * [-6] \\
Q_2[-2] = Q_2Q_0[-1] = Q_0Q_1[-1] = Q_0Q_1[1] * [-8] \\
Q_2(\tilde{Q}_2[1] * [-2]) = Q_2(Q_2[1] * [-2]) + Q_2(c^\vee * [-1]) \\
= Q_2Q_2([1]) * Q_0Q_0([-1]) + Q_1Q_2([1]) * Q_1Q_0([-1]) + Q_0Q_2([1])*Q_2Q_0([-1]) \\
+ Q_2(c^\vee) *Q_0([-1]) + Q_1(c^\vee) * Q_1([-1]) + Q_0(c^\vee) * Q_2([-1]) \\
= Q_2Q_2([1]) * [-4] + Q_0Q_2([1])*Q_0Q_1([-1]) +Q_2(c^\vee) *[-2] \\
+ Q_1(c^\vee)*Q_1([-1]) + Q_0(c^\vee) * Q_2([-1]) \\
= Q_2Q_2([1]) * [-4] + Q_0Q_2([1])*Q_0Q_1([1])*[-8] + Q_2(c^\vee)*[-2] \\
+ Q_1(c^\vee)*Q_1([1])*[-4] + c^\vee * c^\vee * Q_2([1])* [-4] + c^\vee * c^\vee * Q_0Q_1([1]) * [-6] \\
= (\gamma_2^2 \odot 1\dip{*})^\vee + ((\gamma_1\dip{2})^3 \odot 1\dip{*})^\vee + ((\gamma_1\dip{2})^2 \odot \gamma_1 \odot 1\dip{*})^\vee \\
+ (\gamma_1^2 c\dip{2} \odot 1\dip{*})^\vee + (\gamma_1 c\dip{2} \odot \gamma_1 \odot 1\dip{*})^\vee + (c\dip{2} \odot \gamma_1\dip{2})^\vee.
\end{gather*}

\textbf{Claim 2}. $ \langle (c\dip{3} \odot 1\dip{*}), \tilde{Q}_2(\tilde{Q}_2([1])*[-2]) \rangle = 1 $. As $ \langle (c\dip{3} \odot 1\dip{*}), Q_2(\tilde{Q}_2([1]) * [-2]) \rangle = 0 $ by Claim 1, this completes the proof.

By Proposition \ref{prop:E2 same} and the Adem relations, $ \tilde{Q}_2 Q_1 = 0 $ and $ \tilde{Q}_2 Q_0 = Q_0 Q_1 $.

By the compatibility between the Kudo--Araki operations and the Hopf algebra structure on $ H_*(D_\infty(BU(1)); \mathbb{F}_2) $ we obtain that:
\begin{align*}
\tilde{Q}_2(\tilde{Q}_2([1])*[-2]) &= \tilde{Q}_2\tilde{Q}_2([1]) * [-4] + Q_1\tilde{Q}_2([1]) * Q_1Q_0([-1]) + \\
Q_0\tilde{Q}_2([1]) * \tilde{Q}_2Q_0([-1]) &= \tilde{Q}_2\tilde{Q}_2([1]) * [-4] + Q_0\tilde{Q}_2([1])*Q_0Q_1([1])*[-8].
\end{align*}

We deduce from the previous equality and from the definition of stabilized Hopf monomials that
\begin{align*}
\langle c\dip{3} \odot 1\dip{*}, \tilde{Q}_2(\tilde{Q}_2([1])*[-2]) \rangle &= \langle c\dip{3} \odot 1\dip{*}, \tilde{Q}_2 \tilde{Q}_2([1])*[-4] \rangle \\
+ \langle c\dip{3} \odot 1\dip{*}, Q_0\tilde{Q}_2([1])*Q_0Q_1([1])*[-8] \rangle &= \langle c\dip{3} \odot 1_1, \tilde{Q}_2\tilde{Q}_2([1]) \rangle \\
+ \langle c\dip{3} \odot 1_5, Q_0\tilde{Q}_2([1])*Q_0Q_1([1]) \rangle.
\end{align*}

A direct calculation using the Hopf monomial basis shows that
\[
\langle c\dip{3} \odot 1_5, Q_0\tilde{Q}_2([1])*Q_0Q_1([1]) \rangle = 0.
\]
Therefore, it is enough to prove that
\[
\langle c\dip{3} \odot 1_1, \tilde{Q}_2\tilde{Q}_2([1]) \rangle = 1.
\]

In order to do so, we use geometric cohomology and we argue as done for Lemma \ref{lem:identity unstable}. Choose $ M = \overline{\Conf}_4(\mathbb{R}^\infty; \mathbb{CP}^\infty) $ as a model for $ D_4(BU(1)) $.
Let $ W \subseteq \frac{\Conf_4(\mathbb{R}^\infty; \mathbb{CP}^\infty)}{\Sigma_1 \times \Sigma_3} $ be the finite-codimensional submanifold consisting of configuration of four points $ p_1,\dots,p_4 $ with labels $ x_1,\dots,x_4 $ where the first coordinate of $ x_1 $ is $ 0 $. Let $r_W \colon W \to M $ be the composition of the embedding of $ W $ inside $ \frac{\Conf_4(\mathbb{R}^\infty; \mathbb{CP}^\infty)}{\Sigma_1 \times \Sigma_3} $ and the quotient map to $ M $.
$ c\dip{3} \odot 1_1 $ is represented by the manifold over $ M $ $ r_W \colon W \to M $, in the sense made explicit in the proof of Claim 2.

Moreover, $ \tilde{Q}_2\tilde{Q}_2([1]) $ is represented by the compact submanifold $ \widetilde{\Orb}_{(2,2)} $, diffeomorphic to $ S^2 \times_{\Sigma_2} (\mathbb{RP}^2)^2 $, of configuration of points of the form $ \nu_{2;2,2} \circ (\id_S \times_{\Sigma_2} \nu_{2;1,1}^2) (S \times_{\Sigma_2} (S \times_{\Sigma_2} \{*\})^2) $, where $ S $ is the subspace of configuration of two antipodal points on $ S^2 $ and $ * \in BU(1) $ is any chosen basepoint.

If we choose $ * = [1:0:0:\dots] $, a direct calculation shows that, by identifying $ S^2 $ with $ \mathbb{CP}^1 = \mathbb{C} \cup \{\infty\} $,
\begin{gather*}
\widetilde{\Orb}_{(2,2)} = \{\left[((q+\frac{1}{2}p,q-\frac{1}{2}p,-q+\frac{1}{2}p,-q-\frac{1}{2}p),\right. \\
([pq:-(p+q):1:0:0:\dots],[q:-(1+\overline{p}q):\overline{p}:0:0:\dots],\\
[p:-(1+p\overline{q}):\overline{q}:0:0:\dots], [1:-(\overline{p}+\overline{q}):\overline{p}\overline{q}:0:0:\dots])) \Bigg]:p.q \in S^2\}.
\end{gather*}

This compact submanifold intersects the image of $ r_W $ in a single non-singular point, thus
\[
\langle c\dip{3} \odot 1_1, \tilde{Q}_2\tilde{Q}_2([1]) \rangle = 1.
\]
\end{proof}

\begin{remark}
With some more work involving the computation of pairing between some geometric homology and geometric cohomology representatives, one can explicitly calculate $ \tilde{Q}_2 \left( \tilde{Q}_2([1]) *[-2] \right) $ and obtain the identity
\begin{align*}
\tilde{Q}_2\left( \tilde{Q}_2([1])*[-2] \right) &= ((\gamma_1\dip{2})^2 \odot 1\dip{*})^\vee + (c\dip{2} \odot \gamma_1\dip{2} \odot 1\dip{*})^\vee + ((\gamma_1^2) \odot 1\dip{*})^\vee \\
&+ ((\gamma_1\dip{2})^3 \odot 1\dip{*})^\vee + (\gamma_1c\dip{2} \odot \gamma_1 \odot 1\dip{*})^\vee \\
&+ (c\dip{3} \odot 1\dip{*})^\vee + (c^2 \odot c \odot 1\dip{*})^\vee + (c\dip{3} \odot 1\dip{*})^\vee.
\end{align*}
We will not dig into the lengthy calculations, that can be done as in the proof of the previous proposition, because we do not need them in the remaining part of this paper.
\end{remark}

\section{The $ E_3 $-structure on unordered flag manifolds and its non-extendibility.}

The mod $ 2 $ (co)homology of complete unordered flag manifolds is computed by Guerra and Jana in \cite{Guerra-Santanil}. In that paper it is shown that $ H^*(\overline{\Fl}_\infty(\mathbb{C}); \mathbb{F}_2) $ is the quotient of $ H^*(D_\infty(BU(1)); \mathbb{F}_2) $ by the ideal generated by the Chern classes of a certain fiber bundle. In particular, the restriction map $ H^*(D_\infty(BU(1)); \mathbb{F}_2) \to H^*(\overline{\Fl}_\infty(\mathbb{C}); \mathbb{F}_2) $ is surjective. Dually, the map in homology $ H_*(\overline{\Fl}_\infty(\mathbb{C}); \mathbb{F}_2) \to H_*(D_\infty(BU(1)); \mathbb{F}_2) $ is injective.

Assuming the $ E_3 $-Atiyah conjecture, the inclusion $ \overline{\Fl}_\infty(\mathbb{C}) \hookrightarrow D_\infty(BU(1)) $ is an $ E_3 $-map, hence its induced map in homology preserves Dyer--Lashof operations. Consequently, injectivity in homology and Proposition \ref{prop:different} imply the following result.

\begin{corollary}
Assume that the $ E_3 $-Atiyah conjecture is true. Then the classical and Atiyah $ E_3 $-structures on the group completion of $ \bigsqcup_n \overline{\Fl}_n(\mathbb{C}) $ are different (and non-homotopic).
\end{corollary}

We also discuss the possibility of extending these structures to higher order homotopies.
\begin{proposition} \label{prop:non-extendible}
The Atiyah $ E_3 $-action on $ \bigsqcup_n \overline{\Fl}_n(\mathbb{C}) $ does not extend to an $ E_4 $-action. The Atiyah $ E_2 $-action on $ \bigsqcup_n \overline{\Fl}_n(\mathbb{R}) $ does not extend to an $ E_3 $-action.
\end{proposition}
\begin{proof}
We recall that $ U(1) = K(\mathbb{Z},1) $, hence $ BU(1) = K(\mathbb{Z};2) $. Therefore $ \pi_2(BU(1)) \cong \mathbb{Z} $. If we realize $ BU(1) $ as $ \mathbb{CP}^\infty $, a free generator for $ \pi_2(BU(1)) $ is the class of the inclusion map $ j \colon S^2 \cong \mathbb{CP}^1 \hookrightarrow \mathbb{CP}^\infty $.

Let $ i \colon S^2 \to \Conf_2(\mathbb{R}^3) \subseteq \FM_3(2) $ be defined by $ i(p) = (p,-p) $.
We Let $ * $ be a basepoint in $ BU(1) $ and we consider the continuous map given by the composition
\[
\varphi \colon S^2 \stackrel{i}{\to} \Conf_2(\mathbb{R}^3) \stackrel{\nu_{2;1,1}(\cdot,*,*)}{\to} D_2(BU(1)) = \overline{\Conf}_2(\mathbb{R}^\infty; BU(1)).
\]

We are now going to show that $ [\varphi] \in \pi_2(D_2(BU(1))) $ is non-zero. We observe that the covering $ p_1 \colon \Conf_2(\mathbb{R}^\infty; BU(1)) \to \overline{\Conf}_2(\mathbb{R}^\infty; BU(1)) $ induces isomorphisms of higher homotopy groups. $ \nu_{2;1,1}|_{\Conf_2(\mathbb{R}^3)} $ lifts to
$ \tilde{\nu}_{2;1,1} \colon \Conf_2(\mathbb{R}^3) \to \Conf_2(\mathbb{R}^\infty; BU(1)) $ by using the same definition and keeping track of the order of points. Consequently, $ \varphi $ lifts to a map $ \tilde{\varphi} \colon S^2 \to \Conf_2(\mathbb{R}^\infty; BU(1)) $.
Therefore, it is enough to prove that $ [\tilde{\varphi}] \not= 0 \in \pi_2(\Conf_2(\mathbb{R}^\infty;BU(1))) $.

Since $ \Conf_2(\mathbb{R}^\infty) $ is contractible, the projection
\[
p_2 \colon \Conf_2(\mathbb{R}^\infty; BU(1)) \to BU(1)^2 = (\mathbb{CP}^\infty)^2
\]
induces an isomorphism on homotopy groups. Therefore, it is enough to check that
\[
[p_2 \circ \tilde{\varphi}] \not= 0\in \pi_2((\mathbb{CP}^\infty)^2) \cong \mathbb{Z} \times \mathbb{Z}.
\]

Let $ \sigma \colon \mathbb{CP}^1 \to \mathbb{CP}^1 $ be the map given by $ \sigma(z) = -\overline{z}^{-1} $. Note that $ \sigma $ is a topological involution of $ \mathbb{CP}^1 $, hence $ \pi_2(\sigma) \colon \pi_2(\mathbb{CP}^1) \to \pi_2(\mathbb{CP}^1) $ is an isomorphism.
A direct calculation shows that $ p_2 \circ \tilde{\varphi} = j \times (j \circ \sigma) $. This defines a non-zero class in $ \pi_2(\mathbb{CP}^\infty)^2 $, hence $ [\varphi] \not= 0 $.

We now complete our proof by arguing by contradiction. We assume that the Atiyah $ E_3 $-structure extends to an $ E_4 $-structure. Then $ \nu_{2;1,1} $ would extend to a map
\[
FM_4^2 \times BU(1)^2 \to D_2(BU(1)),
\]
and thus $ \varphi $ would extend to $ S^3 $, in which $ S^2 $ can be contracted to a point. But this is impossible because $ [\varphi] $ is non-zero in $ \pi_2(D_2(BU(1)) $, as we have just proved.
\end{proof}

\section{Stable cohomology of unordered flag manifolds and stable extendibility}

Proposition \ref{prop:non-extendible} does not imply that the $ E_3 $ (respectiverly $ E_2 $) structure on the group completion of unordered complex (respectively real) flag manifolds does not extend to an $ E_4 $ (respectively $ E_3 $), or even $ E_\infty $, action. The answer of this latter question is currently unknown.

In this section, we compute the action of Dyer--Lashof operations on the mod $ 2 $ homology of stable unordered flag manifolds, under the assumption of the $ E_3 $ Atiyah--Sutcliffe conjecture. Our calculations show that mod $ 2 $ (co)homology does not see obstructions for the $ E_3 $-action to stably extend to an $ E_\infty $-one.

\begin{proposition}
Let $ \tilde{Q}_i^* \colon H^*(\overline{\Fl}_\infty(\mathbb{C}); \mathbb{F}_2) \to H^{*-i}(\overline{\Fl}_\infty(\mathbb{C}); \mathbb{F}_2) $ be the linear dual of the Kudo--Araki--Dyer--Lashof operations (in lower indices notation) associated with the Atiyah $ E_3 $ structure. For $ i = 1,2 $, $ \tilde{Q}_i^* $ is zero on the ideal $ I $ generated by pullbacks of Chern classes.
\end{proposition}
\begin{proof}
We consider the map $ p_n \colon D_n(BU(1)) \to BU(n) $ induced by the standard inclusion of $ \Sigma_n \wr U(1) $ into $ U(n) $.
We note that $ \bigsqcup_n p_n $ is an $ E_\infty $-map, inducing an $ \infty $-loop space map on the group completions $ p \colon Q(BU(1)) \to BU $.
$ p^* \colon H^*(BU(\infty); \mathbb{F}_2) \to H^*(Q_0BU(1); \mathbb{F}_2) $ preserves the operations $ Q_i^* $ and is a Hopf algebra map. In particular,
\[
\tilde{Q}_1^*(p^*(c_n)) = Q_1^*(p^*(c_n)) = p^*(Q_1^*(c_n)) = 0,
\]
the last equality being true by degree reasons.
Hence, $ Q_1^* $ is $ 0 $ on the ideal $ I $.

Regarding $ \tilde{Q}_2^* $, we recall that it satisfies the Cartan formula with respect to the coproduct in cohomology. We have that
\begin{equation} \label{eq:coproduct pullback Chern}
\begin{split}
\Delta(p^*(c_n)) &= (p^* \otimes p^*)(\Delta(c_n)) = (p^* \otimes p^*)(\sum_{i=0}^n c_i \otimes c_{n-i}) \\
&= \sum_{i=0}^n p^*(c_i) \otimes p^*(c_{n-i}).
\end{split}
\end{equation}

Using Equation~\ref{eq:coproduct pullback Chern}, we prove by induction on $ n $ that $ \tilde{Q}_2^*(p^*(c_n)) = 0 $.
For $ n = 1 $ (base of the induction), we prove by direct inspection that $ \tilde{Q}_2^*(p^*(c_1)) = 0 $. Indeed, by dimensional reasons, $ \tilde{Q}_2^*(p^*(c_1)) $ is a multiple of the dual class of $ 1 \in H_0(Q_0(BU(1)); \mathbb{F}_2) $, but $ \tilde{Q}_2(1) = 0 $.

We now assume inductively that $ \tilde{Q}_2^*(p^*(c_k)) = 0 $ for all $ k < n $ and we prove that $ \tilde{Q}_2^*(p^*(c_n)) = 0 $. By equation~\ref{eq:coproduct pullback Chern} and the Cartan formula, $ \tilde{Q}_2^*(p^*(c_n)) $ is primitive of degree $ 2n-2 $. Therefore, from \cite{GSS} we deduce that it must be a linear combination of (stabilizations of) primitive decorated gathered blocks
\[
b = \prod_{i=1}^N (\gamma_i^{[2^{N-i}]})^{a_i} \cdot ((c_1^l)^{[2^N]})^{a_0},
\]
with $ a_0,\dots,a_{N-1} \geq 0 $ and $ a_N > 0 $.

By dualizing the Adem relation $ \tilde{Q}_2 \circ Q_1 = 0 $, we obtain that $ Q_1^* \circ \tilde{Q}_2^*(c_n) = 0 $. This implies that a gathered block $ b $ as above can appear as an addend of $ \tilde{Q}_2^*(c_n) $ only if $ a_N \geq 2 $. Otherwise, it would pair non-trivially with an indecomposable element $ Q_I(x) $, where $ Q_I $ is an admissible sequece of Kudo--Araki operations starting with $ Q_1 $.

Moreover, since the restriction of $ p^*(c_n) $ to the components lower than $ n $ is $ 0 $, only gathered blocks $ b $ with $ 2^N \geq \frac{n}{2} $ can appear.

There exists gathered blocks satisfying this conditions and having the right cohomological degree if and only if $ n = 2^{N+1} $ for some $ N \geq 0 $, and in this case the only admissible gathered block is $ \gamma_{N}^2 $. Explicitly, there exists $ \lambda_N \in \mathbb{F}_2 $ such that
\[
\tilde{Q}_2^*(p^*(c_n)) = \left\{ \begin{array}{ll}
\lambda_N \gamma_N^2 & \mbox{if } n = 2^{N+1} \\
0 & \mbox{otherwise}
\end{array} \right. .
\]

In the case $ n = 2^{N+1} $, in order to prove that $ \lambda_N = 0 $, we observe that, by applying Wu's formula and dualizing the Nishida relations,
\begin{align*}
\Sq^{2^N} \circ \tilde{Q}_2^*(p^*(c_{2^{N+1}})) &= \tilde{Q}_2^* \circ \Sq^{2^{N+1}}(p^*(c_{2^{N+1}})) \\
&= \tilde{Q}_2^* \left( \sum_{k=0}^N p^*(c_{2^{N+1} + 2^k}) \cdot p^*(c_{2^N-2^k}) \right) \\
&= \sum_{k=0}^N \tilde{Q}_2^*(p^*(c_{2^{N+1}+2^k})) \cdot \tilde{Q}_0^*(p^*(c_{2^N-2^k})) = 0
\end{align*}
because, since $ 2^{N+1} + 2^k $ is never a power of $ 2 $, $ \tilde{Q}_2^*(p^*(c_{2^{N+1}+2^k})) = 0 $.

In contrast, by the Steenrod algebra calculations in the cohomology of the symmetric groups \cite{Sinha:12},
\[
\Sq^{2^N}(\gamma_N^2) = (\Sq^{2^{N-1}}(\gamma_N))^2 = \gamma_N^2 \cdot (\gamma_1^{[2^{N-1}]})^2,
\]
which is not $ 0 $. This forces $ \lambda_N = 0 $.
\end{proof}

The stabilization of complete unordered flag manifolds $ \overline{\Fl}_\infty(\mathbb{C}) $ is the fiber of the $ \infty $-loop space map $ p \colon Q(BU(1)) \to BU $. The inclusion $ \overline{\Fl}_\infty(\mathbb{C}) \rightarrow Q_0(BU(1)) $ is an $ E_3 $-map. Thus, it preserves the homology operations $ Q_i $ and, consequently, their duals $ Q_i^* $.
From \cite{Guerra-Santanil}, $ H^*(\overline{\Fl}_\infty(\mathbb{C}); \mathbb{F}_2) $ is the quotient of $ H^*(Q_0(BU(1)); \mathbb{F}_2) $ by the ideal $ I $ generated by $ p^*(c_1), \dots, p^*(c_n), \dots $, the pullbacks under $ p $ of the mod $ 2 $ reduction of the universal Chern classes $ c_n \in H^{2n}(BU; \mathbb{F}_2) $. Therefore, the proposition above implies that, algebraically, the (dual) Dyer--Lashof operations on $ Q_0(BU(1)) $ fully determine the (dual) Dyer--Lashof operations on $ \overline{\Fl}_\infty(\mathbb{C}) $.

%\clearpage
	
\bibliographystyle{plain}
\bibliography{bibliografia}
	
\end{document}